\documentclass[final,leqno]{siamltex}
\usepackage{amsmath}
\usepackage{graphicx}
\usepackage{float}
 \usepackage[notcite,notref]{showkeys}
\usepackage{mathrsfs}
\usepackage{graphics}
\usepackage{float}
\usepackage{amsfonts,amssymb}
\usepackage{dsfont}
\usepackage{pifont}
\usepackage{wrapfig} 
\usepackage{hyperref}
\usepackage{multirow}
\usepackage{color}
\numberwithin{equation}{section}
\usepackage{threeparttable}
\def\3bar{{|\hspace{-.02in}|\hspace{-.02in}|}}
\def\E{{\mathcal{E}}}
\def\T{{\mathcal{T}}}

\def\pT{{\partial T}}

\def\W{{\mathcal{W}}}

\def\bw{{\mathbf{w}}}
\def\bu{{\mathbf{u}}}

\def\bv{{\mathbf{v}}}

\def\bn{{\mathbf{n}}}

\def\bx{{\mathbf{x}}}
\def\bq{{\mathbf{q}}}

\def\be{{\mathbf{e}}}
\def\bw{{\mathbf{w}}}
\def\bf{{\mathbf{f}}}
\def\bQ{{\mathbf{Q}}}
\def\bphi{{\boldsymbol{\phi}}}
\def\bzeta{{\boldsymbol{\zeta}}}

\def\bphi{{\boldsymbol{\phi}}}
\def\ljump{{[\![}}
\def\rjump{{]\!]}}

\newtheorem{remark}{Remark}[section]
\newtheorem{algorithm}{Stabilizer-Free Weak Galerkin Algorithm}[section]

\def\ad#1{\begin{aligned}#1\end{aligned}}  \def\b#1{{\mathbf{#1}}}
\def\a#1{\begin{align*}#1\end{align*}} \def\an#1{\begin{align}#1\end{align}} 
 
\def\p#1{\begin{pmatrix}#1\end{pmatrix}}

\title{
Stabilizer-free  Weak Galerkin Methods   for  Quad-Curl Problems on polyhedral  Meshes without Convexity Assumptions 
}

\begin{document}
 \author{
 Chunmei Wang \thanks{Department of Mathematics, University of Florida, Gainesville, FL 32611, USA (chunmei.wang@ufl.edu). The research of Chunmei Wang was partially supported by National Science Foundation Grant DMS-2136380.}
  \and
 Shangyou Zhang\thanks{Department of Mathematical Sciences,  University of Delaware, Newark, DE 19716, USA (szhang@udel.edu).  }
 }

\maketitle
\begin{abstract} This paper introduces an efficient stabilizer-free weak Galerkin (WG) finite element method for solving the three-dimensional quad-curl problem. Leveraging bubble functions as a key analytical tool, the method extends the applicability of stabilizer-free WG approaches to non-convex elements in finite element partitions—a notable advancement over existing methods, which are restricted to convex elements. The proposed method maintains a simple, symmetric, and positive definite formulation. It achieves optimal error estimates for the exact solution in a discrete norm, as well as an optimal-order
$L^2$
  error estimate for 
$k>2$ and a sub-optimal order for the lowest order case 
$k=2$. Numerical experiments are presented to validate the method's efficiency and accuracy.

\end{abstract}



\pagestyle{myheadings}

\section{Introduction}
In this paper, we focus on developing a stabilizer-free weak Galerkin finite element method specifically tailored to address the quad-curl problem in three dimensions without convexity constraints. The objective is to find 
$\bu$
 such that, for a given $\bf$ defined on a bounded domain $\Omega\subset \mathbb R^3$,  
\begin{equation}\label{model} 
\begin{split}
 (\nabla \times)^4 \bu=&\bf, \qquad \text{in}\quad \Omega,\\
 \nabla\cdot\bu=&0, \qquad \text{in}\quad \Omega,\\
 \bu\times\bn=&0, \qquad \text{on}\quad \partial\Omega,\\
\nabla\times\bu\times\bn=&0, \qquad \text{on}\quad \partial\Omega.
 \end{split}
\end{equation}

Quad-curl problems arise in various scientific applications, including inverse electromagnetic scattering theory for nonhomogeneous media \cite{13} and magnetohydrodynamics equations \cite{55}. In recent years, significant progress has been made in developing finite element methods to address these problems. Conforming finite element spaces have been constructed for quad-curl problems in both two dimensions (e.g., \cite{28,51}) and three dimensions (e.g., \cite{27,39,52}). Nonconforming and low-order finite element spaces were proposed in \cite{30,55}, while mixed finite element methods were introduced in \cite{sun2,49,53}. The application of Hodge decomposition to quad-curl problems was investigated in \cite{10}, and a discontinuous Galerkin method was developed in \cite{25}. A novel weak Galerkin formulation, leveraging conforming spaces for curl-curl problems as nonconforming spaces for quad-curl problems, was proposed in \cite{sun}. Moreover, a posteriori error analysis for two-dimensional quad-curl problems was carried out in \cite{50}, and a virtual element method for two-dimensional quad-curl problems was introduced in \cite{54}. In addition, \cite{cao} presented a decoupled formulation for quad-curl problems, accompanied by both a priori and a posteriori error analyses.

The weak Galerkin (WG) finite element method represents a significant advancement in numerical techniques for solving partial differential equations (PDEs). This approach approximates differential operators within a framework analogous to the theory of distributions for piecewise polynomials. Unlike traditional finite element methods, WG relaxes the regularity requirements for approximating functions by incorporating carefully designed stabilizers. Extensive research has demonstrated the versatility of the WG method across a wide range of model PDEs, as evidenced by numerous studies \cite{wg1, wg2, wg3, wg4, wg5, wg6, wg7, wg8, wg9, wg10, wg11, wg12, wg13, wg14, wg15, wg16, wg17, wg18, wg19, wg20, wg21, itera, wy3655}, solidifying its position as a powerful tool in scientific computing.
What sets WG methods apart from other finite element approaches is their reliance on weak derivatives and weak continuities to design numerical schemes directly based on the weak forms of PDEs. This structural flexibility enables WG methods to handle a wide spectrum of PDEs effectively, ensuring both stability and accuracy in their solutions.

A notable innovation within the WG framework is the ``Primal-Dual Weak Galerkin (PDWG)'' method, which addresses challenges that traditional numerical techniques often struggle to overcome \cite{pdwg1, pdwg2, pdwg3, pdwg4, pdwg5, pdwg6, pdwg7, pdwg8, pdwg9, pdwg10, pdwg11, pdwg12, pdwg13, pdwg14, pdwg15}. PDWG formulates numerical solutions as constrained minimizations of functionals, where the constraints reflect the weak formulation of PDEs through weak derivatives. This approach leads to an Euler-Lagrange system involving both the primal variable and a dual variable (Lagrange multiplier), resulting in a symmetric and robust numerical scheme.
 
The WG methods for quad-curl problems discussed in the literature vary in their approaches and applicability. The method introduced in \cite{sun} is curl-conforming and specifically designed for tetrahedral partitions. In contrast, the WG method proposed in \cite{wg9} does not require curl-conformity, allowing it to be applied to arbitrary polyhedral partitions. This method has been shown to deliver accurate and reliable solutions for the quad-curl system, achieving optimal error estimates in discrete norms and $L^2$ error estimates, except for the lowest-order case ($k=2$). Furthermore, numerical experiments in \cite{wg9} have demonstrated notable superconvergence phenomena, further highlighting its effectiveness.

This paper presents a stabilizer-free WG finite element method for three-dimensional quad-curl problems. By eliminating the need for stabilizers, the proposed method streamlines implementation and applies to both convex and non-convex elements in finite element partitions, leveraging bubble functions as a key analytical tool. Previous stabilizer-free WG methods have been confined to convex polytopal meshes \cite{ye1, ye2}. While recent advancements have introduced stabilizer-free WG methods without convexity assumptions for other problems, including the Poisson equation, biharmonic equation, and elasticity problems \cite{wang1, wang2, wangelas, wangelas2}, this work is the first to extend such approaches to quad-curl problems without requiring convexity assumptions.

The proposed method maintains the size and global sparsity of the stiffness matrix, reducing programming complexity compared to stabilizer-dependent methods. Moreover, it leverages bubble functions without imposing the restrictive conditions required by other stabilizer-free WG approaches \cite{ye1, ye2}, allowing for broader applicability to various PDEs without additional implementation challenges.
Theoretical analysis confirms optimal error estimates for WG approximations in discrete norms. Additionally, the method achieves an optimal-order $L^2$ error estimate for $k > 2$ and a suboptimal-order $L^2$ error estimate for the lowest-order case $k = 2$. Numerical experiments validate the theoretical findings, showcasing the efficiency and practical utility of the proposed method.

The paper is organized as follows: Section 2 derives a weak formulation for the quad-curl system \eqref{model}, followed by Section 3, which reviews the discrete weak gradient and discrete weak curl-curl operators. Section 4 presents the stabilizer-free weak Galerkin  algorithm for the quad-curl problem. In Section 5, the existence and uniqueness of the solution for the WG scheme are derived, while Section 6 focuses on deriving the error equations for the WG scheme. Section 7 establishes optimal order error estimates in discrete norms for the WG approximation. Section 8 provides the 
$L^2$
  error estimate for the WG solution, which is established in an optimal order except for the lowest order 
$k=2$ under certain regularity assumptions. Finally, Section 9 demonstrates the numerical performance of the WG algorithm through test examples.

We follow the standard notations for Sobolev spaces and norms defined on a given open and bounded domain  $D\subset \mathbb{R}^3$ with Lipschitz continuous boundary.  Let $\|\cdot\|_{s,D}$, $|\cdot|_{s,D}$ and $(\cdot,\cdot)_{s,D}$ denote the norm, seminorm, and inner product in the Sobolev space $H^s(D)$ for any $s\ge 0$. The space $H^0(D)$ coincides with $L^2(D)$, the space of square-integrable functions, for which the norm and inner product are denoted by  $\|\cdot \|_{D}$ and $(\cdot,\cdot)_{D}$	
 , respectively. When $D=\Omega$ or when the domain of integration is clear from the context, we shall omit the subscript $D$ in the norm and the inner product notation. The symbol 
  $C$ denotes a generic constant independent of the meshsize and other physical or functional parameters.

\section{A Weak Formulation}\label{Section:2}
Let $s>0$ be an integer. We first introduce the space:
$$
H(curl^s; \Omega)=\{\bu\in [L^2(\Omega)]^3: (\nabla\times)^j\bu\in [L^2(\Omega)]^3, j=1, \cdots, s\} 
$$
with the associated inner product 
  $$
 (\bu, \bv)_{H(curl^s; \Omega)}=(\bu, \bv)+\sum_{j=1}^s ((\nabla\times)^j\bu,(\nabla\times)^j\bv)
  $$
  and the norm $$
 \|\bu\|_{H(curl^s; \Omega)}= (\bu, \bu)^{\frac{1}{2}}_{H(curl^s; \Omega)}.$$  
 Next, we introduce the following spaces:
$$
H_0(curl; \Omega):=\{\bu\in H(curl; \Omega): \bn\times \bu=0\ \text{on}\  \partial\Omega\},
$$
$$
H_0(curl^2; \Omega):=\{\bu\in H(curl^2; \Omega): \bn\times \bu=0\ \text{and}\ \nabla\times\bu\times \bn=0\ \text{on}\ \partial \Omega\}.
$$
Additionally, we introduce:
$$
H(div; \Omega)=\{\bu\in [L^2(\Omega)]^3: \nabla\cdot\bu \in L^2(\Omega)\},
$$
with the associated inner product 
$$(\bu, \bv)_{H(div; \Omega)}=(\bu, \bv)+(\nabla\cdot\bu, \nabla\cdot\bv)$$
and the norm 
$$
\|\bu\|_{H(div; \Omega)}= (\bu, \bu)^{\frac{1}{2}}_{H(div; \Omega)}. 
$$
Finally, we introduce:
$$
H(div^0; \Omega)=\{\bu\in H(div;\Omega): \nabla\cdot\bu=0\ 
 \text{in} \ \Omega\}.
$$

By utilizing the integration by parts, we propose the following weak formulation of the quad-curl problem \eqref{model}:  Given $\bf\in H(div^0; \Omega)$, seek  $(\bu; p)\in H_0(curl^2; \Omega)\times H_0^1(\Omega)$ such that
\begin{equation}\label{weakform}
\begin{split}
((\nabla\times)^2\bu, (\nabla\times)^2\bv)+(\bv, \nabla p)=&(\bf, \bv), \qquad \bv \in H_0(curl^2; \Omega),\\
\alpha(\nabla p, \nabla q)-(\bu, \nabla q)=&0, \qquad \qquad \forall q\in H_0^1(\Omega),
\end{split}
\end{equation}
where $\alpha$ is  a function to be specified later.  

\begin{theorem} \cite{sun}
Given $\bf\in H(div^0; \Omega)$, the problem \eqref{weakform} has a unique solution $(\bu; p)\in H_0(curl^2; \Omega)\times H_0^1(\Omega)$. Furthermore, $p=0$ and $\bu$ satisfies
$$
\|\bu\|_{H(curl^2; \Omega)}\leq C\|\bf\|.
$$

\end{theorem}

\section{Discrete Weak Differential Operators}\label{Section:Hessian}
The principal differential operators in the weak formulation (\ref{weakform}) for the quad-curl problem (\ref{model}) are the gradient operator $\nabla$ and the curl-curl operator $(\nabla \times)^2$. In this section, we will briefly review the discrete weak gradient operator as described in \cite{pdwg9, wg9} and the discrete weak curl-curl operator as discussed in \cite{wg9}.

Let $T$ be a polyhedral domain with boundary $\partial T$. A scalar-valued weak function on $T$ is denoted by $\sigma=\{\sigma_0,\sigma_b\}$, where 
 $\sigma_0\in L^2(T)$ and $\sigma_b\in L^{2}(\partial T)$. Here, $\sigma_0$ and $\sigma_b$  represent the values of $\sigma$ in the interior and on the boundary of $T$, respectively. Note that $\sigma_b$ is not necessarily the trace of $\sigma_0$ on $\partial T$. Denote by $\W(T)$ the space of scalar-valued weak functions on $T$, defined as:
\begin{equation*}\label{2.1}
\W(T)=\{\sigma=\{\sigma_0,\sigma_b\}: \sigma_0\in L^2(T), \sigma_b\in
L^{2}(\partial T)\}.
\end{equation*}

A vector-valued weak function on $T$ refers to a triplet $\bv=\{\bv_0,\bv_b, \bv_n\}$,  where  $\bv_0\in [L^2(T)]^3$ and $\bv_b\in [L^2(\partial T)]^3$ and $\bv_n\in [L^2(\partial  T)]^3$. Here $\bv_0$ and $\bv_b$  represent the values of $\bv$ in the interior and on the boundary of $T$, while $\bv_n$ represents the value of $\nabla\times\bv$ on $\partial T$. Note that $\bv_b$ and $\bv_n$ are not necessarily the traces of $\bv_0$ and $\nabla \times \bv_0$ on $\partial T$ respectively. 

Denote by $V(T)$ the space of vector-valued weak functions on $T$:
\begin{equation}\label{V} \ad {
V(T)=\{\bv=\{\bv_0, & 
  \bv_b=v_1\b t_1+v_2\b t_2, \bv_n=v_3\b t_1+v_4\b t_2\}: \\ 
 & \bv_0\in [L^2(T)]^3, \; v_i \in L^{2}(e), \; i=1,\dots 4 \},
  }
\end{equation} where $\b t_1$ and $\b t_2$ are two unit tangent vectors
  on a face polygon $e$ of $T$.
 
The weak gradient of $\sigma\in \W(T)$, denoted by $\nabla_w \sigma$, is defined as a linear functional on $[H^1(T)]^3$ such that
\begin{equation*}
(\nabla_w  \sigma,\boldsymbol{\psi})_T=-(\sigma_0,\nabla \cdot \boldsymbol{\psi})_T+\langle \sigma_b,\boldsymbol{\psi}\cdot \textbf{n}\rangle_{\partial T},
\end{equation*}
for all $\boldsymbol{\psi}\in [H^1(T)]^3$.

The weak curl-curl operator of  $\bv\in V(T)$, denoted by $(\nabla\times)^2_{w}\bv$, is defined in the dual space of $H(curl^2; T)$. Its action on $\bq\in H(curl^2; T)$ is given by 
$$
((\nabla \times)^2_{w} \bv, \bq)_T=(\bv_0, (\nabla\times)^2 \bq)_T-\langle \bv_b\times\bn, \nabla\times\bq\rangle_{\partial T}-\langle \bv_n\times\bn, \bq\rangle_{\partial T}.
$$

Let 
 $P_r(T)$ denote the space of polynomials on $T$ with degree no more than $r$. 

A discrete version of $\nabla_{w}\sigma$  for $\sigma\in \W(T)$, denoted by $\nabla_{w, r, T}\sigma$, is defined as a unique polynomial vector in $[P_r(T) ]^3$ satisfying
\begin{equation}\label{disgradient}
(\nabla_{w, r, T} \sigma, \boldsymbol{\psi})_T=-(\sigma_0, \nabla \cdot \boldsymbol{\psi})_T+\langle \sigma_b, \boldsymbol{\psi} \cdot \textbf{n}\rangle_{\partial T}, \quad\forall\boldsymbol{\psi}\in [P_r(T)]^3,
\end{equation}
 which, from the usual integration by parts, gives
 \begin{equation}\label{disgradient*}
 (\nabla_{w, r, T} \sigma, \boldsymbol{\psi})_T= (\nabla \sigma_0, \boldsymbol{\psi})_T-\langle \sigma_0- \sigma_b, \boldsymbol{\psi} \cdot \textbf{n}\rangle_{\partial T}, \quad\forall\boldsymbol{\psi}\in [P_r(T)]^3,
 \end{equation}
 provided that $\sigma_0\in H^1(T)$.

A discrete version of $(\nabla\times)^2_{w}\bv$ for $\bv\in V(T)$, denoted by $(\nabla\times)^2_{w, r, T}\bv$, is defined as a unique polynomial vector in $[P_r(T) ]^3$ satisfying
\begin{equation}\label{discurlcurl}
((\nabla \times)^2_{w, r, T} \bv, \bq)_T=(\bv_0, (\nabla\times)^2 \bq)_T-\langle \bv_b\times\bn, \nabla\times\bq\rangle_{\partial T}-\langle \bv_n\times\bn, \bq\rangle_{\partial T}, 
\end{equation}
 for any $\bq \in [P_r(T)]^3$, which, from the usual integration by parts, yields
\begin{equation}\label{discurlcurlnew}
\begin{split}
&((\nabla \times)^2_{w, r, T} \bv, \bq)_T\\=&((\nabla\times)^2 \bv_0,   \bq)_T-\langle (\bv_b-\bv_0)\times\bn, \nabla\times\bq\rangle_{\partial T}-\langle (\bv_n-\nabla\times\bv_0)\times\bn, \bq\rangle_{\partial T}, 
\end{split} 
\end{equation}
 for any $\bq \in [P_r(T)]^3$, provided that $\bv_0\in H(curl^2; T)$.

\section{Stabilizer-Free Weak Galerkin Algorithm}\label{Section:WGFEM}
Let ${\cal T}_h$ be a finite element partition of the domain $\Omega\subset\mathbb R^3$, consisting of  shape-regular  polyhedra  \cite{wy3655}. Denote by ${\mathcal E}_h$ the set of all faces in ${\cal T}_h$, and by ${\mathcal E}_h^0={\mathcal E}_h \setminus
\partial\Omega$ the set of all interior faces. Let $h_T$ denote  the meshsize of $T\in {\cal T}_h$ and $h=\max_{T\in {\cal T}_h}h_T$ the meshsize for the partition ${\cal T}_h$.

For any given integer { $k\geq 2$}, denote by
$W_k(T)$ the local discrete space of the scalar-valued weak functions given by
$$
W_k(T)=\{\{\sigma_0,\sigma_b\}:\sigma_0\in P_k(T),\sigma_b\in
P_k(e),e\subset \partial T\}.
$$

  Denote by
$V_k(T)$ the local discrete space of the vector-valued weak functions given by
\an{\label{Vk}\ad{ 
V_k(T)=\{ \bv\in \b V(T):\bv_0\in [P_k(T)]^3,  v_1, v_2\in
P_k(e),  v_3, v_4\in
P_{k-1}(e) \} } } where $V(T)$ is defined in 
 \eqref{V}.

Patching $W_k(T)$ over all the elements $T\in {\cal T}_h$
through a common value $\sigma_b$ on the interior interface $\E_h^0$, we form the following scalar-valued 
weak finite element space, denoted by $W_h$:
$$
W_h=\big\{\{\sigma_0, \sigma_b\}:\{\sigma_0, \sigma_b\}|_T\in W_k(T), \forall T\in {\cal T}_h \big\},
$$
and the subspace of $W_h$ with vanishing boundary values on $\partial\Omega$, denoted by $W_h^0$: 
\begin{equation*}\label{W0}
W_h^0=\{\{\sigma_0, \sigma_b\}\in W_h: \sigma_b=0\ \text{on}\ \partial \Omega\}.
\end{equation*}
  Similarly,  patching $V_k(T)$ over all the elements $T\in {\cal T}_h$
through a common value $\bv_b$ on the interior interface $\E_h^0$, we form the following vector-valued weak finite element space, denoted by $V_h$:
$$
V_h=\big\{\{\bv_0,\bv_b, \bv_n\}:\{\bv_0,\bv_b, \bv_n\}|_T\in V_k(T), \forall T\in {\cal T}_h \big\},
$$
and the subspace of $V_h$ with vanishing boundary values on $\partial\Omega$, denoted by $V_h^0$:
\begin{equation*}\label{V0}
V_h^0=\big\{\{\bv_0,\bv_b, \bv_n\}\in V_h: \bv_b\times\bn=0\  \text{and} \ \bv_n\times\bn=0\ \ \text{on}\ \partial\Omega\big\}.
\end{equation*}

Let $r_1$ and $r_2$ be two integers.
For simplicity of notation and without confusion, for any $\sigma\in
W_h$ and $\bv\in V_h$, denote by $\nabla_{w}\sigma$ and $(\nabla \times) ^2_{w} \bv$ the discrete weak actions   $\nabla_{w, r_2, T}\sigma$ and  $(\nabla \times) ^2_{w, r_1, T} \bv$ computed by using   (\ref{disgradient}) and \eqref{discurlcurl} on each element $T$; i.e.,
$$
(\nabla_{w}\sigma)|_T= \nabla_{w, r_2, T}(\sigma|_T), \qquad \sigma\in W_h,
$$
 $$
({\nabla\times}^2)_{w} \bv|_T=({\nabla\times}^2)_{w, r_1, T}(\bv|_T), \qquad \bv\in V_h.
$$

For any $p, q\in W_h$ and $\bu, \bv\in V_h$, we introduce the
following bilinear forms
\begin{eqnarray*}\label{EQ:local-stabilizer}
a(\bu, \bv)=&\sum_{T\in {\cal T}_h}a_T(\bu, \bv),\\
b(\bu, q)=&\sum_{T\in {\cal T}_h}b_T(\bu, q), \\
c(p, q)=&\sum_{T\in {\cal T}_h}c_T(p, q), 
\label{EQ:local-bterm}
\end{eqnarray*} 
where
\begin{equation*}
\begin{split}
a_T(\bu, \bv) =& ((\nabla\times)_w^2 \bu, (\nabla\times)_w^2 \bv)_T,\\
b_T(\bu,q)=&(\bu_0, \nabla_w q)_T,\\
c_T (p, q)=& h_T^4(\nabla_w p, \nabla_w q)_T.
 \end{split}
\end{equation*} 
 
The following stabilizer-free weak Galerkin scheme for the quad-curl problem (\ref{model}) is based on the variational formulation (\ref{weakform}):
\begin{algorithm}\label{a-1}
Given $\bf \in H(div^0; \Omega)$, find $(\bu_h; p_h)\in V_h^0 \times W_{h}^0$, such that
\begin{eqnarray}\label{32}
 a(\bu_h, \bv_h)+b(\bv_h, p_h)&=& (\bf, \bv_0), \qquad \forall \bv_h\in V_{h}^0,\\
c(p_h, q_h) -b(\bu_h, q_h)&=&0,\qquad \quad \qquad \forall q_h\in W_h^0.\label{2}
\end{eqnarray}
\end{algorithm}

\section{Solution Existence and Uniqueness}

Recall that $\T_h$ is a shape-regular finite element partition of
the domain $\Omega$. For any $T\in\T_h$ and $\varphi\in H^{1}(T)$, the following trace inequality holds true \cite{wy3655}:
\begin{equation}\label{trace-inequality}
\|\varphi\|_{\pT}^2 \leq C
(h_T^{-1}\|\varphi\|_{T}^2+h_T\| \varphi\|_{1, T}^2).
\end{equation}
Furthermore, if $\varphi$ is a polynomial on $T$, the standard inverse inequality yields: 
\begin{equation}\label{trace}
\|\varphi\|_{\pT}^2 \leq Ch_T^{-1}\|\varphi\|_{T}^2.
\end{equation}

 For any $\sigma=\{\sigma_0, \sigma_b\}\in W_h$, we define
the   energy norm as follows: \begin{equation}\label{3norma}
\3bar \sigma\3bar_{W_h}=\Big( \sum_{T\in {\cal T}_h} h_T^4(\nabla_w \sigma, \nabla_w \sigma)_T\Big)^{\frac{1}{2}},
\end{equation}
and the following discrete $H^1$ semi-norm: 
\begin{equation}\label{disnorma}
\|\sigma\|_{1, h}=\Big( \sum_{T\in {\cal T}_h}h_T^4 \|\nabla \sigma_0\|_T^2+h_T^{3}\|\sigma_0-\sigma_b\|_{\partial T}^2\Big)^{\frac{1}{2}}.
\end{equation}
\begin{lemma}\cite{wang1}\label{norm1a}
 For $\sigma=\{\sigma_0, \sigma_b\}\in W_h$, there exists a constant $C$ such that
 $$
 \|\nabla \sigma_0\|_T\leq C\|\nabla_w \sigma\|_T.
 $$
\end{lemma}
\begin{proof} 
 Let  $T\in {\cal T}_h$ be a polyhedral element with $N$ faces denoted by $e_1, \cdots, e_N$. It is important to emphasis that the polyhedral element $T$  can be non-convex. On each  face $f_i$, we construct   a linear equation  $l_i(x)$ such that  $l_i(x)=0$ on  face $f_i$ as follows: 
$$l_i(x)=\frac{1}{h_T}\overrightarrow{AX}\cdot \bn_i, $$  where  $A=(A_1,  A_2)$ is a given point on   face $f_i$,  $X=(x_1, x_2)$ is any point on face $f_i$, $\bn_i$ is the normal direction to   face $f_i$, and $h_T$ is the size of the element $T$. 

The bubble function of  the element  $T$ can be  defined as 
 $$
 \Phi_B =l^2_1(x)l^2_2(x)\cdots l^2_N(x) \in P_{2N}(T).
 $$ 
 It is straightforward to verify that  $\Phi_B=0$ on the boundary $\partial T$.    The function 
  $\Phi_B$  can be scaled such that $\Phi_B(M)=1$ where   $M$  represents the barycenter of the element $T$. Additionally,  there exists a sub-domain $\hat{T}\subset T$ such that $\Phi_B\geq \rho_0$ for some constant $\rho_0>0$.

For $\sigma=\{\sigma_0, \sigma_b\}\in W_h$, letting $r_2=2N+k-1$  and $\boldsymbol{\psi}=\Phi_B \nabla \sigma_0 \in [P_{r_2}(T)]^3$ in \eqref{disgradient*} yields 
\begin{equation}\label{t1a}
\begin{split}
(\nabla_w \sigma, \Phi_B \nabla \sigma_0)_T&=(\nabla \sigma_0, \Phi_B \nabla \sigma_0)_T+\langle \sigma_b-\sigma_0,  \Phi_B \nabla \sigma_0 \cdot \bn\rangle_{\partial T}\\&=(\nabla \sigma_0, \Phi_B \nabla \sigma_0)_T,
\end{split}
\end{equation}
where we used $\Phi_B=0$ on $\partial T$.

From the domain inverse inequality \cite{wy3655},  there exists a constant $C$ such that 
\begin{equation}\label{t2a}
(\nabla \sigma_0, \Phi_B \nabla \sigma_0)_T \geq C (\nabla \sigma_0, \nabla \sigma_0)_T.
\end{equation} 
By applying the Cauchy-Schwarz inequality and using \eqref{t1a}-\eqref{t2a}, we get
 $$
 (\nabla \sigma_0, \nabla \sigma_0)_T\leq C (\nabla_w \sigma, \Phi_B \nabla \sigma_0)_T  \leq C  \|\nabla_w \sigma\|_T \|\Phi_B \nabla \sigma_0\|_T  \leq C
\|\nabla_w \sigma\|_T \|\nabla \sigma_0\|_T,
 $$
which simplifies to
 $$
 \|\nabla \sigma_0\|_T\leq C\|\nabla_w \sigma\|_T.
 $$

This completes the proof of the lemma.
\end{proof}

\begin{remark}
   If the polyhedral element $T$  is convex, 
   the bubble function of  the element  $T$ in Lemma \ref{norm1a}  can be  simplified to
 $$
 \Phi_B =l_1(x)l_2(x)\cdots l_N(x).
 $$ 
It can be verified that there exists a sub-domain $\hat{T}\subset T$,  such that
 $ \Phi_B\geq\rho_0$  for some constant $\rho_0>0$,  and $\Phi_B=0$ on the boundary $\partial T$.   Lemma \ref{norm1a}   can be proved in the same manner using this simplified construction. In this case, we take $r_2=N+k-1$.  
\end{remark}

By constructing a face-based bubble function   $$\varphi_{f_k}= \Pi_{i=1, \cdots, N, i\neq k}l_i^2(x),$$   it can be easily verified that (1) $\varphi_{f_k}=0$ on each  face  $f_i$ for $i \neq k$, (2) there exists a sub-domain $\widehat{f_k}\subset f_k$ such that  $\varphi_{f_k} \geq \rho_1$ for some constant $\rho_1>0$.

\begin{lemma}\label{phi1}
   For $\{\sigma_0, \sigma_b\}\in W_h$, let $\boldsymbol{\psi}=(\sigma_b-\sigma_0)  \varphi_{f_k}\bn$, where $\bn$ is the unit outward normal direction to the face  $f_k$. The following inequality holds:
\begin{equation}
  \|\boldsymbol{\psi}\|_T ^2 \leq Ch_T \int_{f_k}(\sigma_b-\sigma_0)^2ds.
\end{equation}
\end{lemma}
\begin{proof}
 We first extend $\sigma_b$, initially defined on the two dimensional  face  $f_k$, to the entire three dimensional polyhedral element $T$  using  the following formula:
$$
 \sigma_b (X)= \sigma_b(Proj_{f_k} (X)),
$$
where $X=(x_1,x_2,x_3)$ is any point in the  element $T$, $Proj_{f_k} (X)$ denotes the orthogonal projection of the point $X$ onto the  face  $f_k$.

We claim that $\sigma_b$ remains  a polynomial defined on the element $T$ after the extension.  

Let the  face  $f_k$ be defined by two linearly independent vectors $\beta_1, \beta_2$ originating from a point $A$ on the  face  $f_k$. Any point $P$ on the  face  $f_k$ can be parametrized as
$$
P(t_1,   t_2)=A+t_1\beta_1+ t_2\beta_2,
$$
where $t_1,  t_2$ are parameters.

Note that $\sigma_b(P(t_1,   t_2))$ is a polynomial of degree $k$ defined on the  face  $f_k$. It can be expressed as:
$$
\sigma_b(P(t_1,  t_2))=\sum_{|\alpha|\leq k}c_{\alpha}\textbf{t}^{\alpha},
$$
where $\textbf{t}^{\alpha}=t_1^{\alpha_1} t_2^{\alpha_2}$ and $\alpha=(\alpha_1, \alpha_2)$.

For any point $X=(x_1, x_2, x_3)$ in the element $T$, the projection  $Proj_{f_k} (X)$ onto the  face  $f_k$ is the point on $f_k$ that minimizes the distance to $X$. Mathematically, this projection $Proj_{f_k} (X)$ is an affine transformation which can be expressed as 
$$
Proj_{f_k} (X)=A+\sum_{i=1}^{2} t_i(X)\beta_i,
$$
where $t_i(X)$ are the projection coefficients, and $A$ is the origin point on $f_k$. The coefficients $t_i(X)$ are determined  by solving the orthogonality condition:
$$
(X-Proj_{f_k} (X))\cdot \beta_j=0,  \qquad j=1, 2.
$$
This results in a system of linear equations in $t_1(X)$, $t_2(X)$, which  can be solved to yield:
$$
t_i(X)= \text{linear function of} \  X.
$$
Hence, the projection $Proj_{f_k} (X)$ is an affine linear function  of $X$.

We extend the polynomial $\sigma_b$ from the  face  $f_k$ to the entire element $T$ by defining
$$
\sigma_b(X)=\sigma_b(Proj_{f_k} (X))=\sum_{|\alpha|\leq k}c_{\alpha}\textbf{t}(X)^{\alpha},
$$
where $\textbf{t}(X)^{\alpha}=t_1(X)^{\alpha_1} t_{2}(X)^{\alpha_{2}}$. Since $t_i(X)$ are linear functions of $X$, each term $\textbf{t}(X)^{\alpha}$ is a polynomial in $X=(x_1, x_2, x_3)$.
Thus, $\sigma_b(X)$ is a polynomial in the  three dimensional coordinates $X=(x_1, x_2, x_3)$.

 Secondly, let $\sigma_{trace}$ denote the trace of $\sigma_0$ on the  face  $f_k$. We extend $\sigma_{trace}$   to the entire element $T$  using  the following formula:
$$
 \sigma_{trace} (X)= \sigma_{trace}(Proj_{f_k} (X)),
$$
where $X$ is any point in the element $T$, $Proj_{f_k} (X)$ denotes the projection of the point $X$ onto the  face  $f_k$. Similar to the case for $\sigma_b$, $\sigma_{trace}$ remains a polynomial after this extension. 

Let $\boldsymbol{\psi}=(\sigma_b-\sigma_0)  \varphi_{f_k}\bn$. We have
\begin{equation*}
    \begin{split}
\|\boldsymbol{\psi}\|^2_T  =
\int_T\boldsymbol{\psi}^2dT =  &\int_T ((\sigma_b-\sigma_0) (X) \varphi_{f_k}\bn)^2dT\\
\leq &Ch_T \int_{f_k} ((\sigma_b-\sigma_{trace}) (Proj_{f_k} (X)) \varphi_{f_k}\bn)^2ds\\
 \leq &Ch_T \int_{f_k} (\sigma_b-\sigma_0)^2ds,
    \end{split}
\end{equation*} 
where we used the facts that (1) $\varphi_{f_k}=0$ on each  face  $f_i$ for $i \neq k$, (2) there exists a sub-domain $\widehat{f_k}\subset f_k$ such that  $\varphi_{f_k} \geq \rho_1$ for some constant $\rho_1>0$,  and applied the properties of the projection.

 This completes the proof of the lemma.

\end{proof}

\begin{lemma} \cite{wang1}  \label{normmm} There exist two positive constants $C_1$ and $C_2$ such that for any $\sigma=\{\sigma_0, \sigma_b\} \in W_h$, the following inequality holds:
 \begin{equation}\label{normeqa}
 C_1\|\sigma\|_{1, h}\leq \3bar \sigma\3bar_{W_h}  \leq C_2\|\sigma\|_{1, h}.
\end{equation}
\end{lemma}

\begin{proof}  
Recall that a face-based bubble function  is defined by $$\varphi_{f_k}= \Pi_{i=1, \cdots, N, i\neq k}l_i^2(\bx).$$ 
By choosing  $\boldsymbol{\psi}=(\sigma_b-\sigma_0)  \varphi_{f_k}\bn$ in \eqref{disgradient*}, we obtain:
\begin{equation}\label{t3a} 
  (\nabla_{w} \sigma, \boldsymbol{\psi})_T=(\nabla \sigma_0, \boldsymbol{\psi})_T+
  \langle \sigma_b-\sigma_0,  \boldsymbol{\psi}\cdot \bn \rangle_{\partial T}=(\nabla \sigma_0, \boldsymbol{\psi})_T+ \int_{f_k}|\sigma_b-\sigma_0|^2 \varphi_{f_k}ds,
  \end{equation} 
  where we used the facts that   (1) $\varphi_{f_k}=0$ on each  face  $f_i$ for $i \neq k$, (2) there exists a sub-domain $\widehat{f_k}\subset f_k$ such that  $\varphi_{f_k} \geq \rho_1$ for some constant $\rho_1>0$, 

By applying the Cauchy-Schwarz inequality, \eqref{t3a}, the domain inequality \cite{wy3655},  and using Lemma \ref{phi1}, we obtain:
\begin{equation*}
\begin{split}
 \int_{f_k}|\sigma_b-\sigma_0|^2  ds\leq &C \int_{f_k}|\sigma_b-\sigma_0|^2  \varphi_{f_k}ds \\
 \leq & C(\|\nabla_w \sigma\|_T+\|\nabla \sigma_0\|_T)\| \boldsymbol{\psi}\|_T\\
 \leq & {Ch_T^{\frac{1}{2}} (\|\nabla_w \sigma\|_T+\|\nabla \sigma_0\|_T) (\int_{ f_k }|\sigma_0-\sigma_b|^2ds)^{\frac{1}{2}}},
 \end{split}
\end{equation*}
which, from Lemma \ref{norm1a}, gives 
$$
 h_T^{-1}\int_{  f_k }|\sigma_b-\sigma_0|^2  ds \leq C  (\|\nabla_w \sigma\|^2_T+\|\nabla \sigma_0\|^2_T)\leq C\|\nabla_w \sigma\|^2_T.
$$
This, together with Lemma \ref{norm1a},  \eqref{3norma} and \eqref{disnorma}, gives
$$
 C_1\|\sigma\|_{1, h}\leq \3bar \sigma\3bar_{W_h}.
$$

Next, from \eqref{disgradient*}, we have
$$
 (\nabla_{w} \sigma, \boldsymbol{\psi})_T= (\nabla \sigma_0,  \boldsymbol{\psi})_T+
  \langle \sigma_b-\sigma_0,  \boldsymbol{\psi}\cdot \bn \rangle_{\partial T},
$$
which, using Cauchy-Schwarz inequality and  the trace inequality \eqref{trace}, gives 
$$
 \Big|(\nabla_{w} \sigma, \boldsymbol{\psi})_T\Big| \leq \|\nabla \sigma_0\|_T \|  \boldsymbol{\psi}\|_T+
Ch_T^{-\frac{1}{2}}\|\sigma_b-\sigma_0\|_{\partial T} \| \boldsymbol{\psi}\|_{T}.
$$
This yields
$$
\| \nabla_{w} \sigma\|_T^2\leq C( \|\nabla \sigma_0\|^2_T  +
 h_T^{-1}\|\sigma_b-\sigma_0\|^2_{\partial T}),
$$
 and further gives $$ \3bar \sigma\3bar_{W_h}  \leq C_2\|\sigma\|_{1, h}.$$

This completes the proof of the lemma.
 \end{proof}

  \begin{remark}
   If the polyhedral element $T$  is convex, 
  the  face-based bubble function in Lemma  \ref{phi1} and Lemma \ref{normmm} can be  simplified to
$$\varphi_{f_i}= \Pi_{k=1, \cdots, N, k\neq i}l_k(x).$$
It can be verified that (1)  $\varphi_{f_i}=0$ on the  face $f_k$ for $k \neq i$, (2) there exists a subdomain $\widehat{f_i}\subset f_i$ such that $\varphi_{f_i}\geq \rho_1$ for some constant $\rho_1>0$. Lemma  \ref{phi1} and Lemma \ref{normmm}   can be proved in the same manner using this simplified construction.  
\end{remark}

For any $\bv=\{\bv_0, \bv_b, \bv_n\}\in V_h$, we define  the energy norm as follows:
\begin{equation}\label{3norm}
\3bar \bv\3bar_{V_h} =\Big(\sum_{T\in {\cal T}_h} \| (\nabla \times)_w^2\bv_0\|_T^2\Big)^{\frac{1}{2}},
\end{equation}
and the following discrete $H^2$ norm:
\begin{equation}\label{disnorm}
\begin{split}
 \| \bv\|_{2,h} =&\Big(\sum_{T\in {\cal T}_h} \| (\nabla \times)^2\bv_0\|_T^2+h_T^{-3}\|\bv_0\times\bn-\bv_b\times\bn \|^2_{\partial T} \\&+h_T^{-1}\| \nabla\times\bv_0\times\bn-\bv_n\times\bn\|^2_{\partial T}\Big)^{\frac{1}{2}}.   
\end{split}
\end{equation}

 \begin{lemma}\label{n2} 
 For $\bv=\{\bv_0, \bv_b, \bv_n\}\in V_h$, there exists a constant $C$ such that
 \begin{equation}\label{normeq1}
     \|(\nabla\times)^2 \bv_0\|_T\leq C\|(\nabla\times)_w^2 \bv\|_T.
 \end{equation} 
\end{lemma}
\begin{proof} 
Recall that the bubble function of  the element  $T$  is defined as follows:
 $$
 \Phi_B =l^2_1(\bx)l^2_2(\bx)\cdots l^2_N(\bx) \in  P_{2N}(T).
 $$
 For $v=\{\bv_0, \bv_b, \bv_n\}\in V_h$,   letting $r_1=2N+k-2$ and
 $\bq=\Phi_B (\nabla \times)^2 \bv_0$ in \eqref{discurlcurlnew} yields 

\begin{equation} \label{t12}
\begin{split}
&((\nabla \times)^2_{w} \bv, \Phi_B (\nabla \times)^2 \bv_0)_T\\=&((\nabla\times)^2 \bv_0,   \Phi_B (\nabla \times)^2 \bv_0)_T-\langle (\bv_b-\bv_0)\times\bn, \nabla\times(\Phi_B (\nabla \times)^2 \bv_0)\rangle_{\partial T}\\&-\langle (\bv_n-\nabla\times\bv_0)\times\bn, \Phi_B (\nabla \times)^2 \bv_0\rangle_{\partial T}
\\=&((\nabla\times)^2 \bv_0,   \Phi_B (\nabla \times)^2 \bv_0)_T,
\end{split} 
\end{equation}
where we used $\Phi_B=0$ and $\nabla\times \Phi_B=0$ on $\partial T$.

Recall that there exists a sub-domain $\hat{T}\subset T$,  such that
 $ \Phi_B\geq\rho_0$  for some constant $\rho_0>0$.  From the domain inverse inequality \cite{wy3655},  there exists a constant $C$ such that 
\begin{equation}\label{t2}
((\nabla\times)^2 \bv_0,   \Phi_B (\nabla \times)^2 \bv_0)_T \geq C ((\nabla\times)^2 \bv_0,   (\nabla \times)^2 \bv_0)_T.
\end{equation} 

From Cauchy-Schwarz inequality and \eqref{t12}-\eqref{t2}, we have
 \begin{equation*}
     \begin{split}
      ((\nabla\times)^2 \bv_0,    (\nabla \times)^2 \bv_0)_T&\leq C ((\nabla \times)^2_{w} \bv, \Phi_B (\nabla \times)^2 \bv_0)_T \\& \leq C
\|(\nabla \times)^2_{w} \bv\|_T \| \Phi_B (\nabla \times)^2 \bv_0\|_T\\
& \leq C
\|(\nabla \times)^2_{w} \bv\|_T \|  (\nabla \times)^2 \bv_0\|_T,
     \end{split}
 \end{equation*}
which gives
 $$
 \|(\nabla\times)^2 \bv_0\|_T\leq C\|(\nabla \times)^2_{w} \bv\|_T.
 $$

This completes the proof of the lemma.
\end{proof}

 
Recall that a face-based bubble function  is as follows:
$$\varphi_{f_k}= \Pi_{i=1, \cdots, N, i\neq k}l_i^2(x).$$    Let $\bq=(\bv_b-\bv_0)\times\bn l_k(x)\varphi_{f_k}$. It is straightforward to check that $\bq=0$ on each face $f_i$ for $i=1, \cdots, N$, $\nabla \times \bq =0$ on each  face  $f_i$ for $i \neq k$ and $\nabla \times \bq =(\bv_b-\bv_0)\times\bn (\nabla \times l_k) \varphi_{f_k}=\mathcal{O}( \frac{(\bv_b-\bv_0)\times\bn \varphi_{f_k}}{h_T})$  on face $f_k$.

\begin{lemma}\label{phi2}
 For $\{\bv_0,\bv_b, \bv_n\}\in V_h$,  let $\bq=(\bv_b-\bv_0)\times\bn l_k(x)\varphi_{f_k}$, where $\bn$ is the unit outward normal direction to the  face  $f_k$. The following inequality holds:
\begin{equation}
  \|\bq\|_T ^2 \leq Ch_T  \int_{f_k}((\bv_b-\bv_0)\times\bn)^2ds.
\end{equation}
\end{lemma}
\begin{proof}
 We first extend $\bv_b$, initially defined on the  two dimensional  face  $f_k$, to the entire three dimensional polyhedral element $T$  using  the following formula:
$$
 \bv_b (X)= \bv_b(Proj_{f_k} (X)),
$$
where $X=(x_1,x_2,x_3)$ is any point in the  element $T$, $Proj_{f_k} (X)$ denotes the orthogonal projection of the point $X$ onto the  face  $f_k$.
We assert that $\bv_b$ remains  a polynomial defined on the element $T$ after the extension. This assertion can be derived in the same manner as in Lemma \ref{phi1}.

 Secondly, let $\bv_{trace}$ denote the trace of $\bv_0$ on the  face  $f_k$. We extend $\bv_{trace}$   to the entire element $T$  using  the following formula:
$$
 \bv_{trace} (X)= \bv_{trace}(Proj_{f_k} (X)),
$$
where $X$ is any point in the element $T$, $Proj_{f_k} (X)$ denotes the projection of the point $X$ onto the  face  $f_k$. Similar to Lemma \ref{phi1}, $\bv_{trace}$ remains a polynomial after this extension. 

Let $\bq=(\bv_b-\bv_0)\times\bn l_k(x)\varphi_{f_k}$. We have
\begin{equation*}
    \begin{split}
\|\bq\|^2_T  =
\int_T \bq^2dT  \leq  &
 C h_T^2\int_T (\nabla\times\bq)^2dT  \\ \leq 
&Ch_T^2 \int_T ( \nabla\times((\bv_b-\bv_{trace})(X)\times\bn l_k\varphi_{f_k}))^2dT\\
\leq &Ch_T^2 \int_T (  (\bv_b-\bv_{trace})(X)\times\bn (\nabla\times l_k)\varphi_{f_k})^2dT\\
\leq &C  \int_T (  (\bv_b-\bv_{trace})(X)\times\bn \varphi_{f_k})^2dT\\
\leq &C h_T \int_{f_k} ((\bv_b-v_{trace})(Proj_{f_k} (X))\times\bn)^2ds\\
 \leq &Ch_T  \int_{f_k} ((\bv_b-\bv_0)\times\bn)^2ds, 
    \end{split}
\end{equation*} 
where we used  the  Poincare inequality since   $\bq=0$ on each face $f_i$ for $i=1, \cdots, N$,  $\nabla \times \bq =0$ on each  face  $f_i$ for $i \neq k$, $\nabla \times \bq =(\bv_b-\bv_0)\times\bn (\nabla \times l_k) \varphi_{f_k}=\mathcal{O}( \frac{(\bv_b-\bv_0)\times\bn \varphi_{f_k}}{h_T})$  on face $f_k$, the fact that there exists a subdomain $\widehat{f_k}$ such that $\varphi_{f_k}\geq \rho_1$ for some constant $\rho_1>0$, and applied the properties of the projection.

 This completes the proof of the lemma.

\end{proof}

\begin{lemma}\label{phi3}
 For $\{\bv_0,\bv_b, \bv_n\}\in V_h$,  let $\bq=(\bv_n-\nabla\times\bv_0)\times\bn \varphi_{f_k}$, where $\bn$ is the unit outward normal direction to the  face  $f_k$. The following inequality holds:
\begin{equation}
  \|\bq\|_T ^2 \leq Ch_T \int_{f_k}((\bv_n-\nabla\times\bv_0)\times\bn )^2ds.
\end{equation}
\end{lemma}

  \begin{proof}
 We first extend $\bv_n$, initially defined on the  two dimensional  face  $f_k$, to the entire three dimensional polyhedral element $T$  using  the following formula:
$$
 \bv_n (X)= \bv_n(Proj_{f_k} (X)),
$$
where $X=(x_1,x_2,x_3)$ is any point in the  element $T$, $Proj_{f_k} (X)$ denotes the orthogonal projection of the point $X$ onto the  face  $f_k$.
We assert that $\bv_n$ remains  a polynomial defined on the element $T$ after the extension. This assertion can be derived in the same manner as in Lemma \ref{phi1}.

 Secondly, let $\bv_{trace}$ denote the trace of $\bv_0$ on the  face  $f_k$. We extend $\bv_{trace}$   to the entire element $T$  using  the following formula:
$$
 \bv_{trace} (X)= \bv_{trace}(Proj_{f_k} (X)),
$$
where $X$ is any point in the element $T$, $Proj_{f_k} (X)$ denotes the projection of the point $X$ onto the  face  $f_k$. Similar to Lemma \ref{phi1}, $\bv_{trace}$ remains a polynomial after this extension. 

Let $\bq=(\bv_n-\nabla\times\bv_0)\times\bn \varphi_{f_k}$. We have
\begin{equation*}
    \begin{split}
\|\bq\|^2_T  =
\int_T \bq^2dT   \leq 
&C  \int_T ( (\bv_n-\nabla\times\bv_{trace})(X)\times\bn \varphi_{f_k})^2dT\\ 
\leq &Ch_T  \int_{f_k} ((\bv_n-\nabla\times\bv_{trace})(Proj_{f_k} (X))\times\bn \varphi_{f_k})^2ds\\
 \leq &Ch_T \int_{f_k} (((\bv_n-\nabla\times\bv_0) \times\bn)^2ds, 
    \end{split}
\end{equation*} 
where we used   the fact that there exists a subdomain $\widehat{f_k}$ such that $\varphi_{f_k}\geq \rho_1$ for some constant $\rho_1>0$,  and applied the properties of the projection.

 This completes the proof of the lemma.

\end{proof}

\begin{lemma}\label{normeqva}   There exist two positive constants $C_1$ and $C_2$ such that for any $\bv=\{\bv_0, \bv_b, \bv_n\} \in V_h$, we have
 \begin{equation}\label{normeq}
 C_1\|\bv\|_{2, h}\leq \3bar \bv\3bar_{V_h}  \leq C_2\|\bv\|_{2, h}.
\end{equation}
\end{lemma} 

\begin{proof}   
   Letting $\bq=(\bv_b-\bv_0)\times\bn l_k(x)\varphi_{f_k}$, it is straightforward to verify that  $\bq=0$ on each face $f_i$ for $i=1, \cdots, N$,  $\nabla \times \bq =0$ on each  face  $f_i$ for $i \neq k$ and $\nabla \times \bq =(\bv_b-\bv_0)\times\bn (\nabla \times l_k) \varphi_{f_k}=\mathcal{O}( \frac{(\bv_b-\bv_0)\times\bn \varphi_{f_k}}{h_T})$  on face $f_k$. Substituting  $\bq=(\bv_b-\bv_0)\times\bn l_k(x)\varphi_{f_k}$ into \eqref{discurlcurlnew} gives:
  \begin{equation} \label{t33}
\begin{split}
&((\nabla \times)^2_{w} \bv, \bq)_T\\=&((\nabla\times)^2 \bv_0,   \bq)_T-\langle (\bv_b-\bv_0)\times\bn, \nabla\times\bq\rangle_{\partial T}-\langle (\bv_n-\nabla\times\bv_0)\times\bn, \bq\rangle_{\partial T}\\
=&((\nabla\times)^2 \bv_0,   \bq)_T-C \int_{f_k} |(\bv_b-\bv_0)\times\bn|^2   (\nabla\times  l_k)\varphi_{f_k} ds\\
=&((\nabla\times)^2 \bv_0,   \bq)_T-C h_T^{-1}\int_{f_k} |(\bv_b-\bv_0)\times\bn|^2   \varphi_{f_k} ds.
\end{split} 
\end{equation}
Using Cauchy-Schwarz inequality, the domain inverse inequality \cite{wy3655}, Lemma \ref{phi2}, and \eqref{t33},  gives
\begin{equation*}
\begin{split}
  \int_{f_k}|(\bv_b-\bv_0)\times\bn|^2 ds\leq & C \int_{f_k} |(\bv_b-\bv_0)\times\bn|^2   \varphi_{f_k}  ds
  \\ \leq& C h_T(\|(\nabla \times)^2_{w} \bv\|_T+\|(\nabla\times)^2 \bv_0\|_T){ \|\bq\|_T}\\
 \leq & C { h_T^{\frac{3}{2}}} (\|(\nabla \times)^2_{w} \bv\|_T+\|(\nabla\times)^2 \bv_0\|_T){ (\int_{f_k} |(\bv_b-\bv_0)\times\bn|^2ds)^{\frac{1}{2}}},
 \end{split}
\end{equation*}
which, from  \eqref{normeq1}, gives 
\begin{equation}\label{t21}
 h_T^{-3}\int_{f_k} |(\bv_b-\bv_0)\times\bn|^2ds \leq C  (\|(\nabla \times)^2_{w} \bv\|_T+\|(\nabla\times)^2 \bv_0\|_T)\leq C\|(\nabla \times)^2_{w} \bv\|^2_T.   
\end{equation}

Letting $\bq=(\bv_n-\nabla\times\bv_0)\times\bn \varphi_{f_k}$ in \eqref{discurlcurlnew}  gives

\begin{equation} \label{t3}
\begin{split}
&((\nabla \times)^2_{w} \bv, \bq)_T\\=&((\nabla\times)^2 \bv_0,   \bq)_T-\langle (\bv_b-\bv_0)\times\bn, \nabla\times\bq\rangle_{\partial T}-\langle (\bv_n-\nabla\times\bv_0)\times\bn, \bq\rangle_{\partial T}\\
=&((\nabla\times)^2 \bv_0,   \bq)_T-\langle (\bv_b-\bv_0)\times\bn, \nabla\times\bq\rangle_{\partial T}\\&-\int_{f_k}| (\bv_n-\nabla\times\bv_0)\times\bn|^2\varphi_{f_k}ds.
\end{split} 
\end{equation}
Using Cauchy-Schwarz inequality, the domain inverse inequality \cite{wy3655}, Lemma \ref{phi3}, the trace inequality \eqref{trace}, the inverse inequality,  \eqref{t21}, and \eqref{t3},   gives
 \begin{equation*}
    \begin{split}
       &  \int_{f_k}| (\bv_n-\nabla\times\bv_0)\times\bn|^2  ds\\\leq &C  \int_{f_k}| (\bv_n-\nabla\times\bv_0)\times\bn|^2  \varphi_{f_k}ds\\
  \leq & C (\|(\nabla \times)^2_{w} \bv\|_T+\|(\nabla\times)^2 \bv_0\|)\| \bq\|_T +  C\|(\bv_b-\bv_0)\times\bn\|_{\partial T}\|\nabla\times\bq\|_{\partial T}\\
 \leq & C h_T^{\frac{1}{2}} (\|(\nabla \times)^2_{w} \bv\|_T+\|(\nabla\times)^2 \bv_0\|)(\int_{f_k}| (\bv_n-\nabla\times\bv_0)\times\bn|^2ds)^{\frac{1}{2}}\\&+ C h_T^{\frac{3}{2}}  \|(\nabla \times)^2_{w} \bv\|_T  h_T^{-\frac{1}{2}}h_T^{-1}
 h_T^{\frac{1}{2}}(\int_{f_k}|(\bv_n-\nabla\times\bv_0)\times\bn|^2ds)^{\frac{1}{2}},  
    \end{split}
\end{equation*}
which, from \eqref{normeq1}, gives 
\begin{equation} \label{t11}
\begin{split}
    h_T^{-1}\int_{f_k}|(\bv_n-\nabla\times\bv_0)\times\bn|^2  ds &\leq C  (\|(\nabla \times)^2_{w} \bv\|_T+\|(\nabla\times)^2 \bv_0\|)\\& \leq C\|(\nabla \times)^2_{w} \bv\|^2_T.
\end{split}
\end{equation}
  Using    \eqref{3norm}-\eqref{normeq1}, \eqref{t21} and  \eqref{t11}   gives
$$
 C_1\|\bv\|_{2, h}\leq \3bar \bv\3bar_{V_h}.
$$

Next, applying Cauchy-Schwarz inequality, the inverse inequality, and  the trace inequality \eqref{trace} to  \eqref{discurlcurlnew}, gives 
\begin{equation*} 
\begin{split}
 \Big| ((\nabla \times)^2_{w} \bv, \bq)_T \Big|  
\leq &\|(\nabla\times)^2 \bv_0\|_T \|\bq\|_T+\|(\bv_b-\bv_0)\times\bn\|_{\partial T}\| \nabla\times\bq\|_{\partial T}\\
 &\qquad
   +\|(\bv_n-\nabla\times\bv_0)\times\bn\|_{\partial T}\| \bq\|_{\partial T}\\ 
\leq &\|(\nabla\times)^2 \bv_0\|_T \|\bq\|_T+Ch_T^{-\frac{3}{2}}\|(\bv_b-\bv_0)\times\bn\|_{\partial T}\|\bq\|_{T} \\
  &\qquad +Ch_T^{-\frac{1}{2}}\|(\bv_n-\nabla\times\bv_0)\times\bn\|_{\partial T}\| \bq\|_{T}. 
\end{split} 
\end{equation*}
This yields
\a{ 
\|(\nabla \times)^2_{w} \bv\|_T^2 & \leq C( \|(\nabla\times)^2 \bv_0\|^2_T  +
 h_T^{-3}\|(\bv_b-\bv_0)\times\bn\|^2_{\partial T} \\
  &\qquad +h_T^{-1}\|(\bv_n-\nabla\times\bv_0)\times\bn\|^2_{\partial T}),
 }
 and further gives $$ \3bar \bv\3bar_{V_h}  \leq C_2\|\bv\|_{2, h}.$$

 This completes the proof of the lemma. 
 \end{proof}

\begin{theorem}
The weak Galerkin finite element scheme (\ref{32})-(\ref{2}) has a unique solution.
\end{theorem}
\begin{proof}
It suffices to prove that $\bf=0$ implies  $\bu_h=0$ and $p_h=0$ in $\Omega$. 
To this end,  we set $\bv_h=\bu_h$ in (\ref{32}) and $q_h=p_h$ in (\ref{2}), yielding:
$$
\sum_{T\in {\cal T}_h} ((\nabla\times)^2_w\bu_h, (\nabla\times)^2_w\bu_h)_T+ h_T^4(\nabla_w p_h, \nabla_w p_h)_T=0.
$$
This leads to the following equalities:
$$
\sum_{T\in {\cal T}_h}((\nabla\times)^2_w\bu_h, (\nabla\times)^2_w\bu_h)_T=0,\qquad
\sum_{T\in {\cal T}_h}h_T^4(\nabla_w p_h, \nabla_w p_h)_T=0,
$$
which, along with \eqref{normeq} and \eqref{normeqa}, implies:
 $$\|\bu_h\|_{2,h}=0, \qquad \|p_h\|_{1,h}=0.$$ 
Consequently, we have:
\begin{eqnarray} 
(\nabla\times)^2\bu_0&=&0, \quad \text{on each}\ T,\label{tt1}
  \\
  \nabla \times \bu_0\times \bn &=&\bu_n\times \bn,   \quad \text{on each}\ \partial T,\label{tt2}\\
   \bu_0 \times \bn&=&\bu_b\times \bn, \quad \text{on each}\ \partial T,\label{tt3}\\
   \nabla p_0&=&0, \quad \text{on each}\  T,\label{tt4}\\
    p_0&=&p_b, \quad \text{on each}\ \partial T.\label{tt5}
 \end{eqnarray}
From \eqref{tt4}, it follows that $p_0=C$ on each $T\in {\cal T}_h$. Coupled with \eqref{tt5}, this implies that $p_0$ is continuous over the domain $\Omega$ and thus $p_0=C$ throughout  $\Omega$. Given that 
 $p_b=0$ on $\partial\Omega$, we deduce that $p_0=p_b=0$ in $\Omega$, and consequently,  $p_h=0$ in $\Omega$.

It follows from \eqref{tt2} and \eqref{tt3} that $\bu_0\times\bn$ and $\nabla\times \bu_0\times \bn$ are continuous across the interior interface ${\cal E}_h^0$. Thus, we have $\bu_0\in H(curl^2; \Omega)$. This, together with \eqref{tt1}, implies that   $(\nabla \times)^2 \bu_0=0$ in $\Omega$. Therefore, there exists a potential function $\phi$ such that  $\nabla \times \bu_0=\nabla \phi$ in $\Omega$. This gives: 
\begin{equation*}
    \begin{split}
  (\nabla \phi, \nabla \phi)  &= \sum_{T\in {\cal T}_h} (\nabla \times \bu_0, \nabla \phi)_T\\
   &= \sum_{T\in {\cal T}_h} ( \bu_0, \nabla \times\nabla \phi)_T +\langle \nabla \phi, \bn\times \bu_0\rangle_{\partial T} 
    \\&= \sum_{T\in {\cal T}_h} \langle \nabla \phi, \bn\times \bu_b\rangle_{\partial T}\\
   &=  \langle \nabla \phi, \bn\times \bu_b\rangle_{\partial \Omega}\\&=0,
   \end{split}
\end{equation*}
where we used the usual integration by parts, \eqref{tt3} and the fact that  $\bn\times \bu_b=0$ on $\partial\Omega$.
This leads to $\phi=C$ in $\Omega$, and thus $\nabla \times \bu_0=0$ in $\Omega$.  Furthermore, there exists a potential function $\psi$ such that $\bu_0=\nabla \psi$ in $\Omega$. 

Recall that we have established $p_h=0$ in the domain $\Omega$ which implies  $c(p_h,q_h)=0$ in \eqref{2}. From  \eqref{disgradient} and \eqref{2}, we have
\begin{equation}\label{ee1}
\begin{split}
0&=\sum_{T\in {\cal T}_h}  (\nabla_w q_h, \bu_0)_T\\
&=\sum_{T\in {\cal T}_h} -(q_0, \nabla\cdot\bu_0)_T+\langle q_b, \bu_0\cdot\bn\rangle_{\pT}\\
&=\sum_{T\in {\cal T}_h}  -(q_0, \nabla\cdot\bu_0)_T+\sum_{e\in {\cal E}_h^0}\langle q_b, \ljump \bu_0\cdot\bn\rjump \rangle_{e},
\end{split}
\end{equation}
where $\ljump \bu_0\cdot\bn\rjump$ denotes  the jump of $\bu_0\cdot\bn$ across the face  $e\in {\cal E}_h^0$ and we have  used $q_b=0$ on $\partial\Omega$.
By setting $q_0=0$ and $q_b=\ljump \bu_0\cdot\bn \rjump$ in \eqref{ee1}, we obtain that $\ljump \bu_0\cdot\bn \rjump=0$ on $e\in {\cal E}_h^0$,  which means $\bu_0\cdot\bn$ is continuous across  the interior interface $e\in {\cal E}_h^0$. Consequently, $\bu_0\in H(div; \Omega)$. By taking  $q_0=\nabla\cdot \bu_0$ and $q_b=0$  in \eqref{ee1}, we deduce that $\nabla \cdot \bu_0=0$ on each $T$, and hence 
 $\nabla \cdot \bu_0=0$ in $\Omega$ due to $\bu_0\in H(div; \Omega)$. Recall that there exists a potential function $\psi$ such that $\bu_0=\nabla \psi$ in $\Omega$. Thus, the equation $\nabla\cdot\bu_0=\Delta \psi=0$   holds strongly in $\Omega$ with the boundary condition $\nabla \psi \times \bn=\bu_0\times \bn=0$ on $\partial \Omega$. This implies that $\psi=C$ in $\Omega$. Therefore,  $\bu_0=\nabla \psi=0$ in $\Omega$. Using \eqref{tt2} and \eqref{tt3}, we deduce that  $\bu_b\times \bn=0$ and $\bu_n\times \bn=0$ in $\Omega$. Hence, we conclude that $\bu_h=0$ in $\Omega$.

This completes the proof of the theorem. 
\end{proof}

\section{Error Equations}\label{Section:error-equation}
This section aims to derive the error equations for the weak Galerkin method (\ref{32})-(\ref{2}) applied to the quad-curl problem \eqref{model}. These equations are essential for the subsequent convergence analysis.

Let  $k\geq 2$. Let $\bQ_0$ be the $L^2$ projection operator onto $[P_k(T)]^3$. Analogously, for $e\subset\partial T$, denote by $\bQ_b$ and $\bQ_n$ the $L^2$ projection operators onto $[P_{k}(e)]^3$ and $[P_{k-1}(e)]^3$, respectively. For $\bw\in [H(curl; \Omega)]^3$, define the $L^2$ projection $\bQ_h \bw\in V_h$ as follows:
$$
\bQ_h\bw|_T=\{\bQ_0 \bw, \bQ_b \bw\times\b n, \bQ_n(\nabla\times\bw)\times\b n\}.
$$
For $\sigma \in H^1(\Omega)$, the $L^2$ projection $Q_h \sigma\in W_h$ is defined by 
 $$
 Q_h\sigma|_T=\{Q_0 \sigma, Q_b \sigma\},
$$
where $Q_0$ and $Q_b$ are the $L^2$ projection operators onto $P_k(T)$ and $P_k(e)$ respectively.  Denote by ${\cal Q}_h^{r_1}$ and ${\cal Q}_h^{r_2}$  the $L^2$ projection operators onto $P_{r_1}(T)$ and $P_{r_2}(T)$, respectively.

\begin{lemma}\label{Lemma5.1}   The operators   ${\cal Q}^{r_1}_h$ and ${\cal Q}^{r_2}_h$ satisfy the following properties, namely: 
\begin{eqnarray}\label{l}
(\nabla\times)^2_w \bw &=& {\cal Q}_h^{r_1}((\nabla\times)^2 \bw), \qquad  \forall \bw\in H(curl^2; T),\\
\nabla_{w}\sigma &=& {\cal Q}^{r_2}_h(\nabla \sigma),  \qquad \qquad \forall  \sigma\in H^1(T). \label{l-2}
\end{eqnarray}
\end{lemma}

\begin{proof}
For any $\bw\in H(curl^2; T)$, it follows from \eqref{discurlcurlnew} that 
\begin{equation*}  
\begin{split}
&((\nabla \times)^2_{w} \bw, \bq)_T\\=&((\nabla\times)^2 \bw,   \bq)_T-\langle (\bw|_{\partial T}-\bw|_T)\times\bn, \nabla\times\bq\rangle_{\partial T} \\
&\qquad -\langle ( (\nabla\times \bw)|_{\partial T}-\nabla\times(\bw|_T))\times\bn, \bq\rangle_{\partial T} 
\\=&((\nabla\times)^2 \bw,   \bq)_T \\
=&({\cal Q}_h^{r_1}((\nabla\times)^2 \bw), \bq)_T,
\end{split} 
\end{equation*}
for any $\bq\in [P_{r_1}(T)]^3$. 
This completes the proof of \eqref{l}. 

For any $\sigma\in H^1(T)$, using \eqref{disgradient*} gives
\begin{equation*} 
 \begin{split}
     (\nabla_{w} \sigma, \boldsymbol{\psi})_T=& (\nabla \sigma, \boldsymbol{\psi})_T-\langle \sigma|_T- \sigma|_{\partial T}, \boldsymbol{\psi} \cdot \textbf{n}\rangle_{\partial T}\\
     =& (\nabla \sigma, \boldsymbol{\psi})_T=({\cal Q}_h^{r_2} \nabla \sigma, \boldsymbol{\psi})_T,
 \end{split} 
 \end{equation*}
 for any $\boldsymbol{\psi}\in [P_{r_2} (T)]^3$. This completes the proof of \eqref{l-2}. 
 
\end{proof}

 Let $(\bu, p)$ be the solution  of \eqref{weakform} and assume that $\bu\in H(curl^4; \Omega)$. Then $(\bu, p)$  satisfies
 \begin{eqnarray} \label{mo1}
 ((\nabla\times)^4\bu, \bv)+(\bv, \nabla p)&=&(\bf, \bv),\\
 (\nabla\cdot\bu, q)&=&0, \label{mo2}
 \end{eqnarray}
for $\bv \in [L^2(\Omega)]^3$ and $q\in L^2(\Omega)$. Let $(\bu_h, p_h)$ be the WG solutions of \eqref{32}-\eqref{2}. Define the error functions $\be_h$ and $\epsilon_h$ by 
\begin{eqnarray}\label{error}
\be_h&=&\bu-
\bu_h,\\
\epsilon_h&=&p-p_h. \label{error-2}
\end{eqnarray}
 \begin{lemma}\label{errorequa}
Let $\bu\in H(curl^4; \Omega)$ be the exact solution of the quad-curl model problem \eqref{model}, and let $(\bu_h; p_h) \in V_h^0\times W_h^0$ be the numerical solution obtained from the WG scheme (\ref{32})-(\ref{2}). The error functions $\be_h$ and $\epsilon_h$, defined in (\ref{error})-(\ref{error-2}), satisfy the following error equations, namely:
\begin{eqnarray}\label{sehv}
a(\be_h, \bv_h)+b(\bv_h, \epsilon_h)&=& \ell_1(\bu, \bv_h),\quad   \forall \bv_h\in V_{h}^0,\\
c(\epsilon_h, q_h)-b(\be_h, q_h)&=&-\ell_2(\bu, q_h),\quad\forall q_h\in W^0_h. \label{sehv2}
\end{eqnarray}
\end{lemma}
Here
 \begin{eqnarray*}\label{lu}
\ell_1(\bu, \bv_h)&=&\sum_{T\in{\cal T}_h}\langle (\bv_0-\bv_b)\times\bn, \nabla\times({\cal Q}_h^{r_1} -I)((\nabla\times)^2\bu)\rangle_{\partial T} \\&&+\langle (\nabla \times \bv_0-\bv_n)\times\bn, ({\cal Q}_h^{r_1} -I)((\nabla\times)^2\bu)\rangle_{\partial T},\\
  \ell_2(\bu, q_h)&=&  \sum_{T\in{\cal T}_h} \langle q_0-q_b, (I-{\cal Q}_h^{r_2})\bu \cdot\bn\rangle_{\pT}.
\end{eqnarray*}

\begin{proof} Using \eqref{l},  and setting $\bq={\cal Q}_h^{r_1} ((\nabla\times)^2\bu)$ in \eqref{discurlcurlnew}, we obtain:
\begin{equation}\label{eq1}
\begin{split}
&\sum_{T\in {\cal T}_h}((\nabla\times)^2_w\bu,  (\nabla\times)^2_w \bv_h)_T\\=& \sum_{T\in {\cal T}_h}({\cal Q}_h^{r_1} ((\nabla\times)^2\bu), (\nabla\times)^2_w \bv_h)_T\\
=& \sum_{T\in {\cal T}_h}((\nabla\times)^2 \bv_0,   {\cal Q}_h^{r_1} ((\nabla\times)^2\bu))_T-\langle (\bv_b-\bv_0)\times\bn, \nabla\times{\cal Q}_h^{r_1} ((\nabla\times)^2\bu)\rangle_{\partial T}\\&-\langle (\bv_n-\nabla\times\bv_0)\times\bn, {\cal Q}_h^{r_1} ((\nabla\times)^2\bu)\rangle_{\partial T} \\ 
=&\sum_{T\in {\cal T}_h} ((\nabla\times)^2 \bv_0,     ((\nabla\times)^2\bu))_T-\langle (\bv_b-\bv_0)\times\bn, \nabla\times{\cal Q}_h^{r_1} ((\nabla\times)^2\bu)\rangle_{\partial T}\\&-\langle (\bv_n-\nabla\times\bv_0)\times\bn, {\cal Q}_h^{r_1} ((\nabla\times)^2\bu)\rangle_{\partial T}.
\end{split}
\end{equation}
Taking $\bv=\bv_0$ in \eqref{mo1}, where $\bv_h=\{\bv_0, \bv_b, \bv_n\} \in V_h^0$ and applying the usual integration by parts, we obtain:
\begin{equation}\label{eq2}
\begin{split}
&\sum_{T\in {\cal T}_h}((\nabla\times)^2\bu, (\nabla\times)^2\bv_0)_T+\langle (\nabla\times)^3\bu, (\bv_0-\bv_b)\times\bn\rangle_{\pT}\\&+\langle (\nabla\times)^2\bu, \nabla\times\bv_0\times\bn-\bv_n\times\bn\rangle_{\pT}+(\nabla p, \bv_0)_T=\sum_{T\in {\cal T}_h} (\bf, \bv_0)_T,
\end{split}
\end{equation}
where we used the facts that
$$
\sum_{T\in {\cal T}_h} \langle (\nabla\times)^2\bu, \bv_n\times\bn\rangle_{\pT}=\langle (\nabla\times)^2\bu, \bv_n\times\bn\rangle_{\partial\Omega}=0,
$$
$$
\sum_{T\in {\cal T}_h} \langle (\nabla\times)^3\bu, \bv_b\times\bn\rangle_{\pT}=\langle (\nabla\times)^3\bu, \bv_b\times\bn\rangle_{\partial\Omega}=0.
$$
Substituting \eqref{eq2} into \eqref{eq1}  gives
\begin{equation}\label{ee3} \begin{aligned} 
& \quad \
\sum_{T\in {\cal T}_h}((\nabla\times)^2_w\bu,  (\nabla\times)^2_w \bv_h)_T\\
&=\sum_{T\in {\cal T}_h}(\bf-\nabla p, \bv_0)_T +\langle (\bv_0-\bv_b)\times\bn, \nabla\times({\cal Q}_h^{r_1} -I)((\nabla\times)^2\bu)\rangle_{\partial T}\\
&\quad \ +\langle (\nabla \times \bv_0-\bv_n)\times\bn, ({\cal Q}_h^{r_1} -I)((\nabla\times)^2\bu)\rangle_{\partial T}.
\end{aligned}
\end{equation}
It follows from \eqref{l-2} that
\begin{equation}\label{eq4}
\begin{split}
b(\bv_h, p)=\sum_{T\in {\cal T}_h}(\nabla_wp, \bv_0)_T=\sum_{T\in {\cal T}_h}({\cal Q}_h^{r_2}(\nabla p), \bv_0)_T=\sum_{T\in {\cal T}_h}(\nabla p, \bv_0)_T.
\end{split}
\end{equation}
Combining \eqref{ee3}-\eqref{eq4} gives
\begin{equation*} 
\begin{split}
&a(\bu, \bv_h)+b(\bv_h, p)\\
=& \sum_{T\in {\cal T}_h}(\bf, \bv_0)_T +\langle (\bv_0-\bv_b)\times\bn, \nabla\times({\cal Q}_h^{r_1} -I)((\nabla\times)^2\bu)\rangle_{\partial T}\\&+\langle (\nabla \times \bv_0-\bv_n)\times\bn, ({\cal Q}_h^{r_1} -I)((\nabla\times)^2\bu)\rangle_{\partial T}.
\end{split}
\end{equation*}
Subtracting \eqref{32} from the above equation gives \eqref{sehv}.

  To derive \eqref{sehv2}, we take 
 $q=q_0$ in \eqref{mo2} and apply  the usual integration by parts,   yielding:
\begin{equation}\label{eq5}
0=-\sum_{T\in {\cal T}_h} (\bu, \nabla q_0)+\sum_{T\in {\cal T}_h} \langle \bu\cdot\bn, q_0-q_b\rangle_{\pT},
\end{equation}
where we used $\sum_{T\in {\cal T}_h} \langle \bu\cdot\bn,  q_b\rangle_{\pT}=\langle \bu\cdot\bn,  q_b\rangle_{\partial\Omega}=0$ due to $q_b=0$ on $\partial\Omega$.

Recall that $p=0$, which implies $c(p,q_h)=0$. Using \eqref{disgradient*} and applying the usual integration by parts, we obtain: 
\begin{equation*} \label{beq}
\begin{split}
c(p,q_h)-b(\bu, q_h) =&\sum_{T\in {\cal T}_h}  -(\bu, \nabla_w q_h)_T\\
=& \sum_{T\in {\cal T}_h}  -({\cal Q}_h^{r_2}\bu, \nabla_w q_h)_T\\
=& \sum_{T\in {\cal T}_h} -(\nabla q_0,  {\cal Q}_h^{r_2}\bu)_T+\langle q_0-q_b,  ({\cal Q}_h^{r_2}\bu)\cdot\bn\rangle_{\pT}\\
=& \sum_{T\in {\cal T}_h} -(\nabla q_0, \bu)_T+\langle q_0-q_b, {\cal Q}_h^{r_2}\bu\cdot\bn\rangle_{\pT}\\
=& \sum_{T\in {\cal T}_h}\langle q_0-q_b, ({\cal Q}_h^{r_2}-I)\bu\cdot\bn\rangle_{\pT},
\end{split}
\end{equation*}
where we used \eqref{eq5} on the last line. 

Subtracting \eqref{2} from the above equation completes the proof of 
\eqref{sehv2}.

This completes the proof of the lemma. 
\end{proof}

%

\section{Error Estimates}\label{Section:error-estimates}
This section aims to establish the error estimate in an energy norm for the weak Galerkin finite element method \eqref{32}-\eqref{2} applied to the quad-curl model problem \eqref{model}.

\begin{lemma}\label{lem2}  Let  $k\geq 2$ and $s\in [1, k]$.
Suppose $\bu\in [H^{k+1}(\Omega)]^3$ and $(\nabla\times)^2\bu\in [H^{k-1}(\Omega)]^3$. Then, for $\bu\in V_h$, $\bv\in V_h$ and $q \in  W_h$, the following estimates hold true; i.e.,
\begin{eqnarray}\label{error1}
|\ell_1(\bu, \bv)|&\leq&    Ch^{s-1}\|(\nabla\times)^2\bu\|_{s-1}\|\bv\|_{2,h},\label{error2}\\
|\ell_2(\bu, q)| &\leq& Ch^{s+1}\|\bu\|_{s+1}\| q\|_{1,h}.\label{error3}
\end{eqnarray}
\end{lemma}
\begin{proof}
It follows from  Cauchy-Schwarz inequality and the trace inequality \eqref{trace-inequality} that 
\begin{equation*}
\begin{split}
&|\ell_1(\bu, \bv)| \\=&|\sum_{T\in{\cal T}_h}\langle (\bv_0-\bv_b)\times\bn, \nabla\times({\cal Q}_h^{r_1} -I)((\nabla\times)^2\bu)\rangle_{\partial T} \\&+\langle (\nabla \times \bv_0-\bv_n)\times\bn, ({\cal Q}_h^{r_1} -I)((\nabla\times)^2\bu)\rangle_{\partial T}|
\\
\leq &  \big(\sum_{T\in{\cal T}_h}h_T^{-3}\| (\bv_0-\bv_b)\times\bn\|^2_{\partial T}\big)^{\frac{1}{2}} \big(\sum_{T\in{\cal T}_h}h_T^3\|\nabla\times({\cal Q}_h^{r_1} -I)((\nabla\times)^2\bu)\|^2_{\partial T}\big)^{\frac{1}{2}} \\&+\big(\sum_{T\in{\cal T}_h}h_T^{-1}\|(\nabla \times \bv_0-\bv_n)\times\bn\|^2_{\partial T}\big)^{\frac{1}{2}} \big(\sum_{T\in{\cal T}_h}h_T \| ({\cal Q}_h^{r_1} -I)((\nabla\times)^2\bu)\|^2_{\partial T}\big)^{\frac{1}{2}}  
\\
\leq &\Big( \big(\sum_{T\in{\cal T}_h}h_T^2\|\nabla\times({\cal Q}_h^{r_1} -I)((\nabla\times)^2\bu)\|^2_{T}+h_T^4\|\nabla\times({\cal Q}_h^{r_1} -I)((\nabla\times)^2\bu)\|^2_{1, T}\big)^{\frac{1}{2}} \\&+ \big(\sum_{T\in{\cal T}_h} \|({\cal Q}_h^{r_1} -I)((\nabla\times)^2\bu)\|^2_{T}+h_T^2\|({\cal Q}_h^{r_1} -I)((\nabla\times)^2\bu)\|^2_{1, T}\big)^{\frac{1}{2}}\Big) \|\bv\|_{2,h}\\
\leq &  Ch^{s-1} \|(\nabla\times)^2 \bu\|_{s-1} \|\bv\|_{2,h}.
\end{split}
\end{equation*}

Similarly, using  Cauchy-Schwarz inequality and the trace inequality \eqref{trace-inequality}  gives
\begin{equation*}
\begin{split}
|\ell_2(\bu, q)|&=|\sum_{T\in{\cal T}_h}\langle q_0-q_b, (I-{\cal Q}_h^{r_2})\bu \cdot\bn\rangle_{\pT}|\\
&\leq \Big(\sum_{T\in{\cal T}_h}  h_T^{3} \|q_0-q_b\|_{\pT}^2\Big)^{\frac{1}{2}}\Big(\sum_{T\in{\cal T}_h}h_T^{-3}  \|(I-{\cal Q}_h^{r_2})\bu \cdot\bn\|_{\pT}^2\Big)^{\frac{1}{2}}\\
&\leq \Big(\sum_{T\in{\cal T}_h} h_T^{-4} \|(I-{\cal Q}_h^{r_2})\bu \cdot\bn\|_{T}^2+ h_T^{-2}\|(I-{\cal Q}_h^{r_2})\bu \cdot\bn\|_{1,T}^2\Big)^{\frac{1}{2}}\| q\|_{1,h} \\
&\leq Ch^{s-1}\|\bu\|_{s+1}\| q\|_{1,h}.
  \end{split}
\end{equation*}

This completes the proof of the lemma.
\end{proof}

\begin{theorem}\label{thm}
Let $k\geq 2$. Assume the exact solution of the quad-curl model problem \eqref{model} is sufficiently regular such that $\bu\in [H^{k+1}(\Omega)]^3$ and $(\nabla\times)^2\bu\in [H^{k-1}(\Omega)]^3$. Let $(\bu_h, p_h)\in V_h^0\times W_h^0$ be WG approximations obtained from the stabilizer free WG algorithm \eqref{32}-\eqref{2}. The error functions $\be_h$ and $\epsilon_h$, defined \eqref{error} and \eqref{error-2} satisfy the following error estimate:
 \begin{equation}\label{estimate1}
\3bar \be_h\3bar_{V_h}+\3bar \epsilon_h \3bar_{W_h} \leq Ch^{k-1} (\|\bu\|_{k+1}+\|(\nabla \times)^2\bu\|_{k-1}).
\end{equation} 
\end{theorem}
\begin{proof}
Letting $\bv_h=\be_h$ in (\ref{sehv}) and $q_h=\epsilon_h$ in (\ref{sehv2}), and then adding the two equations, we obtain: 
\begin{equation*}
\begin{split}
\3bar \be_h\3bar^2_{V_h}+\3bar \epsilon_h\3bar^2_{W_h}&=\ell_1(\bu, \be_h)-\ell_2(\bu, \epsilon_h).
\end{split}
\end{equation*} 
Using Lemma \ref{lem2} with $s=k$, and applying \eqref{normeq} and \eqref{normeqa}, completes the proof of the theorem.
\end{proof}

\section{$L^2$ Error Estimates}
This section aims to establish the error estimate in an 
$L^2$ norm for the weak Galerkin   finite element method \eqref{32}-\eqref{2} applied to the quad-curl model problem \eqref{model}.
 To this end,  we consider an auxiliary problem of finding $(\bphi; \xi)$ such that:
\begin{equation}\label{dual}
\begin{split}
(\nabla \times)^4 \bphi+\nabla \xi =&\bzeta_0, \qquad \text{in}\ \Omega,\\
\nabla\cdot\bphi =& 0, \qquad\text{in} \ \Omega,\\
\bphi\times\bn=& 0, \qquad\text{on} \ \partial\Omega,\\
\nabla \times\bphi=& 0, \qquad\text{on} \ \partial\Omega,\\
\xi=& 0, \qquad\text{on} \ \partial\Omega,  
\end{split}
\end{equation}
where we define $\bzeta_h=\bQ_h\bu-\bu_h=\{\bzeta_0,\bzeta_b,\bzeta_n\}$ and  recall that $\be_h=\bu-\bu_h=\{\be_0,\be_b,\be_n\}$.  

Let ${t_0}=\min\{k, 3\}$. We further assume the regularity property for the dual problem \eqref{dual} holds true, in the sense that $\bphi$ and $\xi$ satisfy:
 \begin{equation}\label{regu}
 \|\bphi\|_{t_0+1} +\|(\nabla\times)^2\bphi\|_{t_0-1}+\|\xi\|_{1}\leq C\| \bzeta_0\|.
\end{equation}

 \begin{lemma}
Assume the exact solutions $\xi$  and $\bphi$ of the  auxiliary problem  \eqref{dual} are sufficiently regular such that $\xi\in H^{k+1} (\Omega)$ and $\bphi\in [H^{k+1} (\Omega)]^3$.
Then, there exists a constant $C$, such that the following error estimates hold true, namely:
\begin{equation}\label{erroresti1}
\3bar \xi-Q_h\xi \3bar _{W_h}\leq Ch^{k+2}\|\xi\|_{k+1},
\end{equation}
\begin{equation}\label{errorestiphi}
\3bar \bphi-\bQ_h\bphi \3bar _{V_h}\leq Ch^{k-1}\|\bphi\|_{k+1}.
\end{equation}
\end{lemma}
\begin{proof}
Using \eqref{disgradient*} and the trace inequalities \eqref{trace-inequality}-\eqref{trace}, we have,  for any $\bv\in [P_{r_2} (T)]^3$,   
\begin{equation*}
\begin{split}
&\quad\sum_{T\in {\cal T}_h}(\nabla_w(\xi-Q_h\xi), \bv)_T\\
&=\sum_{T\in {\cal T}_h}(\nabla(\xi-Q_0\xi),  \bv)_T+\langle Q_0\xi-Q_b\xi, \bv\cdot\bn\rangle_{\partial T}\\
&\leq \Big(\sum_{T\in {\cal T}_h}\|\nabla(\xi-Q_0\xi)\|^2_T\Big)^{\frac{1}{2}} \Big(\sum_{T\in {\cal T}_h}\|\bv\|_T^2\Big)^{\frac{1}{2}}\\&\quad+ \Big(\sum_{T\in {\cal T}_h} \|Q_0\xi-Q_b\xi\|_{\partial T} ^2\Big)^{\frac{1}{2}}\Big(\sum_{T\in {\cal T}_h} \|\bv\|_{\partial T}^2\Big)^{\frac{1}{2}}\\
&\leq\Big(\ \sum_{T\in {\cal T}_h} \|\nabla(\xi-Q_0\xi)\|_T^2\Big)^{\frac{1}{2}}\Big(\sum_{T\in {\cal T}_h} \|\bv\|_T^2\Big)^{\frac{1}{2}}\\&\quad+\Big(\sum_{T\in {\cal T}_h}h_T^{-1} \|Q_0\xi-\xi\|_{T} ^2+h_T \|Q_0\xi-\xi\|_{1,T} ^2\Big)^{\frac{1}{2}}\Big(\sum_{T\in {\cal T}_h}Ch_T^{-1}\|\bv\|_T^2\Big)^{\frac{1}{2}}\\
&\leq Ch^k\|\xi\|_{k+1}\Big(\sum_{T\in {\cal T}_h} \|\bv\|_T^2\Big)^{\frac{1}{2}}.
\end{split}
\end{equation*}
Letting $\bv=\nabla_w(\xi-Q_h\xi)$ gives 
$$
\sum_{T\in {\cal T}_h}h_T^4(\nabla_w(\xi-Q_h\xi), \nabla_w(\xi-Q_h\xi))_T\leq 
 Ch^{k+2}\|\xi\|_{k+1}\3bar \xi-Q_h\xi \3bar _{W_h}.$$  
  This completes the proof of the estimate \eqref{erroresti1}.
 
  Using \eqref{discurlcurlnew}, Cauchy-Schwarz inequality, the inverse inequality,  and the trace inequalities \eqref{trace-inequality}-\eqref{trace}, we have, for any $\bv\in [P_{r_2} (T)]^2$,  
\begin{equation*}
\begin{split}
&\quad \ \sum_{T\in {\cal T}_h}((\nabla_w\times)^2(\bphi-\bQ_h\bphi), \bq)_T\\
&=\sum_{T\in {\cal T}_h}((\nabla\times)^2( \bphi-\bQ_0\bphi),   \bq)_T
   -\langle (\bQ_0\bphi-\bQ_b\bphi)\times\bn, \nabla\times\bq\rangle_{\partial T}\\
&\quad \ -\langle (\nabla\times \bQ_0\bphi
        -\bQ_n(\nabla\times\bphi))\times\bn, \bq\rangle_{\partial T}\\
&\leq \Big(\sum_{T\in {\cal T}_h}\|(\nabla\times)^2
        ( \bphi-\bQ_0\bphi)\|^2_T\Big)^{\frac{1}{2}}
        \Big(\sum_{T\in {\cal T}_h}\|\bq\|_T^2\Big)^{\frac{1}{2}} \\
  &\quad \  + \Big(\sum_{T\in {\cal T}_h} \|(\bQ_0\bphi- 
      \bQ_b\bphi)\times\bn\|_{\partial T} ^2\Big)^{\frac{1}{2}}  
       \Big(\sum_{T\in {\cal T}_h} \|\nabla\times\bq\|_{\partial T}^2\Big)^{\frac{1}{2}} \\
 &\quad \ + \Big(\sum_{T\in {\cal T}_h} \|(\nabla\times \bQ_0\bphi- 
    \bQ_n(\nabla\times\bphi))\times\bn\|_{\partial T} ^2\Big)^{\frac{1}  
     {2}}\Big(\sum_{T\in {\cal T}_h} \| \bq\|_{\partial T}^2\Big)^{\frac{1}{2}}\\
&\leq \Big(\sum_{T\in {\cal T}_h}\|(\nabla\times)^2( \bphi-   
     \bQ_0\bphi)\|^2_T\Big)^{\frac{1}{2}} 
       \Big(\sum_{T\in {\cal    T}_h}\|\bq\|_T^2\Big)^{\frac{1}{2}}
   + \Big(\sum_{T\in {\cal T}_h} h_T^{-1}\|(\bQ_0\bphi-\bQ_b\bphi) \\
& \quad \
      \times\bn\|_{ T} ^2
      +h_T \|(\bQ_0\bphi-\bQ_b\bphi)\times\bn\|_{1, T} ^2\Big)^{\frac{1}{2}}
   \Big(\sum_{T\in {\cal T}_h} h_T^{-3}\| \bq\|_{  T}^2\Big)^{\frac{1}{2}}
    + \Big(\sum_{T\in {\cal T}_h} h_T^{-1}\\
&\quad \ \|(\nabla\times \bQ_0\bphi-\bQ_n(\nabla\times\bphi))\times\bn\|_{  T} 
  + h_T \|(\nabla\times \bQ_0\bphi-\bQ_n(\nabla\times\bphi))\times\bn\|_{ 1, T} ^2\Big)^{\frac{1}{2}} \\ 
&\quad \ \cdot 
  \Big(\sum_{T\in {\cal T}_h} h_T^{-1}\| \bq\|_{ T}^2\Big)^{\frac{1}{2}}\\
&\leq \Big(\sum_{T\in {\cal T}_h}\|(\nabla\times)^2
   ( \bphi-\bQ_0\bphi)\|^2_T\Big)^{\frac{1}{2}}
   \Big(\sum_{T\in {\cal T}_h}\|\bq\|_T^2\Big)^{\frac{1}{2}}
   + \Big(\sum_{T\in {\cal T}_h} h_T^{-1}\|\bQ_0\bphi-  \bphi \|_{ T} ^2\\
&\quad \ +h_T \| \bQ_0\bphi- \bphi \|_{1, T} ^2\Big)^{\frac{1}{2}}
  \Big(\sum_{T\in {\cal T}_h} h_T^{-3}\| \bq\|_{  T}^2\Big)^{\frac{1}{2}}
  + \Big(\sum_{T\in {\cal T}_h} h_T^{-1}\|\nabla\times \bQ_0\bphi- \nabla\times\bphi \|_{  T} \\
&\quad \ + h_T \|\nabla\times \bQ_0\bphi- \nabla\times\bphi \|_{ 1, T} ^2\Big)^{\frac{1}{2}}
  \Big(\sum_{T\in {\cal T}_h} h_T^{-1}\| \bq\|_{ T}^2\Big)^{\frac{1}{2}}\\
&\leq Ch^{k-1}\|\bphi\|_{k+1}\Big(\sum_{T\in {\cal T}_h} \|\bq\|_T^2\Big)^{\frac{1}{2}}.
\end{split}
\end{equation*}
Letting $\bq=(\nabla_w\times)^2(\bphi-\bQ_h\bphi)$ gives 
$$
\sum_{T\in {\cal T}_h} ((\nabla_w\times)^2(\bphi-\bQ_h\bphi), (\nabla_w\times)^2(\bphi-\bQ_h\bphi))_T\leq 
 Ch^{k-1}\|\bphi\|_{k+1}\3bar \bphi-\bQ_h\bphi \3bar _{V_h}.$$   
 This completes the proof of \eqref{errorestiphi}.
\end{proof}

 \begin{theorem}
 Let {$k\geq 2$} and   ${t_0}=\min\{k, 3\}$. Assume the exact solution $\bu$ of the quad-curl model problem \eqref{model} is sufficiently regular such that $\bu\in [H^{k+1}(\Omega)]^3$ and $(\nabla\times)^2\bu\in [H^{k-1}(\Omega)]^3$. Let $(\bu_h, p_h)\in V_h^0\times W_h^0$ be WG approximations obtained from the stabilizer free WG algorithm \eqref{32}-\eqref{2}. The following $L^2$ error estimate holds: 
 \begin{equation}
      \|\be_0\|  \leq  Ch^{t_0+k-2} (\|\bu\|_{k+1}+\|(\nabla \times)^2\bu\|_{k-1}).
 \end{equation}
 In other words, we achieve a sub-optimal order of convergence for { $k=2$} and an optimal order of convergence for $k\geq 3$.
 \end{theorem}
    
 \begin{proof}
 Using the usual integration by parts and \eqref{l-2},  and letting $\bu=\bphi$ and $\bv_h=\bzeta_h$ in \eqref{eq1},  $\bv_h=\bQ_h\bphi$ in \eqref{sehv}, and  $q_h= Q_h\xi$ in \eqref{sehv2},   we  obtain:  
 \begin{equation}\label{dq}
     \begin{split}
& \|\bzeta_0\|^2\\=&\sum_{T\in {\cal T}_h}((\nabla \times)^4 \bphi+\nabla \xi, \bzeta_0)_T\\
 =&\sum_{T\in {\cal T}_h} ((\nabla \times)^2\bphi, (\nabla \times)^2\bzeta_0)_T+\langle \nabla\times\bzeta_0, \bn\times (\nabla \times)^2\bphi\rangle_{\partial T}\\
 &+\langle \bzeta_0,\bn\times(\nabla \times)^3\bphi\rangle_{\partial T} +(\bzeta_0, {\cal Q}^{r_2}_h \nabla  \xi)_T\\
=&\sum_{T\in {\cal T}_h}((\nabla\times)^2_w \bphi,  (\nabla\times)^2_w \bzeta_h)_T  -\langle (\bzeta_0-\bzeta_b)\times\bn, \nabla\times{\cal Q}_h^{r_1} ((\nabla\times)^2\bphi)\rangle_{\partial T}\\&-\langle (\nabla \times \bzeta_0-\bzeta_n)\times\bn, {\cal Q}_h^{r_1} ((\nabla\times)^2\bphi)\rangle_{\partial T}
\\& +\langle \nabla\times\bzeta_0, \bn\times (\nabla \times)^2\bphi\rangle_{\partial T}+\langle \bzeta_0,\bn\times(\nabla \times)^3\bphi\rangle_{\partial T}  +(\nabla_w  \xi, \bzeta_0)_T \\
=&\sum_{T\in {\cal T}_h}((\nabla\times)^2_w  \bphi,  (\nabla\times)^2_w \bzeta_h)_T    -\langle (\bzeta_0-\bzeta_b)\times\bn, \nabla\times({\cal Q}_h^{r_1}-I) ((\nabla\times)^2\bphi)\rangle_{\partial T}\\&-\langle (\nabla \times \bzeta_0-\bzeta_n)\times\bn,( {\cal Q}_h^{r_1}-I) ((\nabla\times)^2\bphi)\rangle_{\partial T}  + (\nabla_w  \xi, \bzeta_0)_T \\
=&\sum_{T\in {\cal T}_h}((\nabla\times)^2_w  \bphi,  (\nabla\times)^2_w \be_h)_T +((\nabla\times)^2_w \bphi,  (\nabla\times)^2_w (\bQ_h\bu-\bu))_T -\ell_1(\bzeta_h, \bphi) \\&  +\sum_{T\in {\cal T}_h}(\nabla_w   \xi, \be_0)_T +(\nabla_w   \xi, \bQ_0\bu-\bu)_T \\
=&\sum_{T\in {\cal T}_h}((\nabla\times)^2_w  \bQ_h\bphi,  (\nabla\times)^2_w \be_h)_T-((\nabla\times)^2_w  (\bQ_h\bphi-\bphi),  (\nabla\times)^2_w \be_h)_T \\&+((\nabla\times)^2_w \bphi,  (\nabla\times)^2_w (\bQ_h\bu-\bu))_T -\ell_1(\bzeta_h, \bphi) \\&  +\sum_{T\in {\cal T}_h} (\nabla_w  Q_h \xi, \be_0)_T +(\nabla_w   (\xi-Q_h\xi), \be_0)_T +(\nabla_w   \xi, \bQ_0\bu-\bu)_T \\
\\=& -b(\bQ_h\bphi,  \epsilon_h) +    \ell_1(\bu,  \bQ_h\bphi) +\sum_{T\in {\cal T}_h} -((\nabla\times)^2_w  (\bQ_h\bphi-\bphi),  (\nabla\times)^2_w \be_h)_T \\& + ((\nabla\times)^2_w \bphi,  (\nabla\times)^2_w (\bQ_h\bu-\bu))_T  -\ell_1(\bzeta_h, \bphi) \\&+\ell_2(\bu, Q_h \xi) +c(\epsilon_h,  Q_h\xi) + \sum_{T\in {\cal T}_h}  (\nabla_w   (\xi-Q_h\xi), \be_0)_T+(\nabla_w   \xi, \bQ_0\bu-\bu)_T\\
=&\sum_{i=1}^9 I_i,
\end{split}
\end{equation} 
where we used  $\bzeta_b\times\bn=0$ and  $\bzeta_n\times\bn=0$ on $\partial\Omega$.

We will estimate each term $I_i$  for $i=1,\cdots, 9$ on the last line of \eqref{dq} individually.  Recall that $t_0=\min\{k, 3\}$.

For the term $I_1$, using \eqref{disgradient*},  the usual integration by parts, Cauchy-Schwarz inequality, the trace inequality \eqref{trace}, \eqref{normeqa}, \eqref{estimate1}, we have
    \begin{equation*} 
        \begin{split}
       | b(\bQ_h\bphi, \epsilon_h)| =&|(\bQ_0\bphi, \nabla_w\epsilon_h)| \\
             =&|\sum_{T\in {\cal T}_h} (\nabla \epsilon_0,  \bQ_0\bphi)_T-\langle \epsilon_0-\epsilon_b, \bQ_0\bphi \cdot \textbf{n}\rangle_{\partial T}|
           \\
            =& |\sum_{T\in {\cal T}_h} (\nabla \epsilon_0,   \bphi)_T-\langle \epsilon_0-\epsilon_b, \bQ_0\bphi \cdot \textbf{n}\rangle_{\partial T}|\\  =&|\sum_{T\in {\cal T}_h}  -(  \epsilon_0,   \nabla\cdot\bphi)_T+\langle  \epsilon_0,\bphi\cdot\bn\rangle_{\partial T} -\langle \epsilon_0-\epsilon_b, \bQ_0\bphi \cdot \textbf{n}\rangle_{\partial T}|\\
             =&|\sum_{T\in {\cal T}_h} -\langle \epsilon_0 -\epsilon_b, (\bQ_0\bphi-\bphi) \cdot \textbf{n}\rangle_{\partial T}|\\
            \leq & 
\Big(\sum_{T\in {\cal T}_h} h_T^{3}\|\epsilon_0 -\epsilon_b\|_{ \partial T}^2 \Big)^{\frac{1}{2}}\Big(\sum_{T\in {\cal T}_h} h_T^{-3}\|(\bQ_0\bphi-\bphi) \cdot \textbf{n}\|_{\partial T}^2 \Big)^{\frac{1}{2}}\\
\leq & C\|\epsilon_h\|_{1,h} \Big(\sum_{T\in {\cal T}_h} h_T^{-4}\|(\bQ_0\bphi-\bphi) \cdot \textbf{n}\|_{T}^2 +h_T^{-2}\|(\bQ_0\bphi-\bphi) \cdot \textbf{n}\|_{1, T}^2 \Big)^{\frac{1}{2}}\\
\leq & C\3bar \epsilon_h\3bar_{W_h} h^{t_0-1}\|\bphi\|_{t_0+1}\\
\leq &  Ch^{t_0+k-2} (\|\bu\|_{k+1}+\|(\nabla \times)^2\bu\|_{k-1})\|\bphi\|_{t_0+1}.
        \end{split}
    \end{equation*}
where we used the second equality in \eqref{dual}, and the fact that $\sum_{T\in {\cal T}_h}\langle  \epsilon_b, \bphi  \cdot \textbf{n}\rangle_{\partial T}=\langle  \epsilon_b, \bphi  \cdot \textbf{n}\rangle_{\partial \Omega}=0$  due to $\epsilon_b=0$ on $\partial \Omega$.

Regarding to the term $I_2$, using Cauchy-Schwarz inequality, the trace inequality \eqref{trace-inequality}, we have
 \begin{equation*}\label{q3}
 \begin{split}
    &|\ell_1(\bu, \bQ_h\bphi)| \\\leq&\Big|\sum_{T\in{\cal T}_h}\langle (\bQ_0 \bphi-\bQ_b\bphi)\times\bn, \nabla\times({\cal Q}_h^{r_1} -I)((\nabla\times)^2\bu)\rangle_{\partial T} \\&+\langle (\nabla \times \bQ_0 \bphi-\bQ_n(\nabla \times\bphi))\times\bn, ({\cal Q}_h^{r_1} -I)((\nabla\times)^2\bu)\rangle_{\partial T}\Big|
\\  
    \leq& \Big(\sum_{T\in {\cal T}_h} h_T^{-1}\| \bQ_0 \bphi- \bphi \|_{ T}^2+h_T \| \bQ_0 \bphi- \bphi \|_{ 1, T}^2\Big)^{\frac{1}{2}} \\&\cdot \Big(\sum_{T\in {\cal T}_h}h_T^{-1} \|\nabla\times({\cal Q}_h^{r_1} -I)((\nabla\times)^2\bu)\|_{ T}^2+h_T  \|\nabla\times({\cal Q}_h^{r_1} -I)((\nabla\times)^2\bu)\|_{ 1, T}^2\Big)^{\frac{1}{2}}\\
    &+\Big(\sum_{T\in {\cal T}_h} h_T^{-1}\|  \nabla \times \bQ_0 \bphi- \nabla \times\bphi  \|_{ T}^2+h_T \|  \nabla \times \bQ_0 \bphi- \nabla \times\bphi  \|_{ 1, T}^2\Big)^{\frac{1}{2}} \\&\cdot \Big(\sum_{T\in {\cal T}_h}h_T^{-1} \| ({\cal Q}_h^{r_1} -I)((\nabla\times)^2\bu)\|_{ T}^2+h_T  \| ({\cal Q}_h^{r_1} -I)((\nabla\times)^2\bu)\|_{ 1, T}^2\Big)^{\frac{1}{2}}\\
    \leq& Ch^{t_0+k-2}\|(\nabla \times)^2\bu\|_{k-1} \|\bphi\|_{{t_0}+1}.
 \end{split} 
\end{equation*}

Regarding to the term $I_3$, using Cauchy-Schwarz inequality, \eqref{errorestiphi} with $k=t_0$, and \eqref{estimate1}, gives
\begin{equation*}
    \begin{split}
      & |\sum_{T\in {\cal T}_h} ((\nabla\times)^2_w  (\bQ_h\bphi-\bphi),  (\nabla\times)^2_w \be_h)_T|
      \\
      \leq & \Big(\sum_{T\in {\cal T}_h}\|(\nabla\times)^2_w  (\bQ_h\bphi-\bphi)\|_T^2\Big)^{\frac{1}{2}}\Big(\sum_{T\in {\cal T}_h}\|(\nabla\times)^2_w \be_h\|_T^2\Big)^{\frac{1}{2}}\\
      \leq & Ch^{t_0-1}\|\bphi\|_{t_0+1}h^{k-1} (\|\bu\|_{k+1}+\|(\nabla \times)^2\bu\|_{k-1})\\
       \leq & Ch^{k+t_0-2}\|\bphi\|_{t_0+1} (\|\bu\|_{k+1}+\|(\nabla \times)^2\bu\|_{k-1}).
    \end{split}
\end{equation*}

For the term $I_4$,  denote by $Q^{t_0-2}$ a $L^2$ projection onto $P_{t_0-2}(T)$.
Using \eqref{discurlcurl} with $\bq=Q^{t_0-2}(\nabla\times)^2_w \bphi$ yields
\begin{equation*}\label{pro}
    \begin{split}
      &   (Q^{t_0-2}(\nabla\times)^2_w \bphi,  (\nabla\times)^2_w (\bQ_h\bu-\bu))_T  \\ =&(\bQ_{0}\bu-\bu, (\nabla\times)^2 (Q^{t_0-2}(\nabla\times)^2_w \bphi))_T\\&-\langle (\bQ_b\bu-\bu)\times\bn, \nabla\times(Q^{t_0-2}(\nabla\times)^2_w \bphi)\rangle_{\partial T}\\&-\langle (\bQ_n(\nabla\times\bu)-\nabla\times\bu)\times\bn, Q^{t_0-2}(\nabla\times)^2_w \bphi\rangle_{\partial T}\\=&0,
    \end{split} 
\end{equation*} 
where we  used the properties of the projection operators $\bQ_0$, $\bQ_b$, and $\bQ_n$. This, together with   Cauchy-Schwarz inequality, \eqref{errorestiphi} and \eqref{l},   gives
\begin{equation*}
    \begin{split}
      & | \sum_{T\in {\cal T}_h}((\nabla\times)^2_w \bphi,  (\nabla\times)^2_w (\bQ_h\bu-\bu))_T | \\=&| \sum_{T\in {\cal T}_h}( (\nabla\times)^2_w \bphi-Q^{t_0-2}(\nabla\times)^2_w \bphi,  (\nabla\times)^2_w (\bQ_h\bu-\bu))_T | 
      \\\leq & \Big(\sum_{T\in {\cal T}_h}\|( {\cal Q}_h^{r_1}((\nabla\times)^2  \bphi)-Q^{t_0-2}({\cal Q}_h^{r_1}((\nabla\times)^2  \bphi))\|_T^2\Big)^{\frac{1}{2}}\\&\cdot\Big(\sum_{T\in {\cal T}_h}\|(\nabla\times)^2_w (\bQ_h\bu-\bu)\|_T^2\Big)^{\frac{1}{2}}\\
      \leq & Ch^{k-1}\|\bu\|_{k+1}h^{t_0-1}\|(\nabla\times)^2\bphi\|_{t_0-1}\\
      \leq & Ch^{k+t_0-2}\|\bu\|_{k+1} \|(\nabla\times)^2\bphi\|_{t_0-1}.
    \end{split}
\end{equation*}
 
 Regarding to the term $I_5$, using \eqref{error2} with $\bu=\bphi$ and $\bv=\bzeta_h$ and $s=t_0$, \eqref{estimate1}, \eqref{errorestiphi} and \eqref{normeq}, we have
 \begin{equation*}\label{q5}
 \begin{split}
|\ell_1(\bzeta_h, \bphi) |  
   \leq &    Ch^{{t_0}-1}\|(\nabla\times)^2 \bphi\|_{{t_0}-1}\|\bzeta_h\|_{2,h} \\
    \leq  &    Ch^{{t_0}-1}\|(\nabla\times)^2 \bphi\|_{{t_0}-1}\3bar \bzeta_h\3bar_{V_h}\\
    \leq  &    Ch^{{t_0}-1}\|(\nabla\times)^2 \bphi\|_{{t_0}-1}(\3bar \be_h\3bar_{V_h}+\3bar \bQ_h\bu-\bu\3bar_{V_h})\\
     \leq & Ch^{k+t_0-2}\|(\nabla\times)^2 \bphi\|_{t_0-1}  (\|\bu\|_{k+1}+\|(\nabla \times)^2\bu\|_{k-1}).
 \end{split}
 \end{equation*}
 
  Regarding to the term $I_6$, using Cauchy-Schwarz inequality, and the trace inequality \eqref{trace-inequality},  we have
 \begin{equation*}\label{last}
     \begin{split}
        | \ell_2(\bu, Q_h\xi) |\leq &  \Big(\sum_{T\in{\cal T}_h}  \|Q_0\xi-Q_b\xi\|_{\partial T}^2\Big)^{\frac{1}{2}} \Big(\sum_{T\in{\cal T}_h} \|(I-{\cal Q}_h^{r_2})\bu \cdot\bn\|_{\partial T}^2\Big)^{\frac{1}{2}} \\
           \leq &  \Big(\sum_{T\in{\cal T}_h}  h_T^{-1}\|Q_0\xi- \xi\|_{T}^2+h_T \|Q_0\xi- \xi\|_{1, T}^2\Big)^{\frac{1}{2}}\\&\cdot \Big(\sum_{T\in{\cal T}_h} h_T^{-1}\|(I-{\cal Q}_h^{r_2})\bu \cdot\bn\|_{ T}^2+h_T \|(I-{\cal Q}_h^{r_2})\bu \cdot\bn\|_{1, T}^2\Big)^{\frac{1}{2}} \\
           \leq & Ch^{k+1}\|\bu\|_{k+1} \|\xi\|_1.
     \end{split}
    \end{equation*}

  Regarding to the term $I_7$, using Cauchy-Schwarz inequality, \eqref{l-2}, \eqref{estimate1}, \eqref{erroresti1}  with $k=0$,   we have
\begin{equation*}
    \begin{split}
       | c(\epsilon_h, Q_h\xi) | =& |\sum_{T\in {\cal T}_h}h_T^4(\nabla_w\epsilon_h, \nabla_w Q_h\xi)|
\\
\leq &|\sum_{T\in {\cal T}_h}h_T^4(\nabla_w\epsilon_h, \nabla_w (Q_h\xi-\xi))+h_T^4(\nabla_w\epsilon_h, \nabla_w  \xi)|\\
\leq &  \Big(\sum_{T\in{\cal T}_h}   h_T^4 \| \nabla_w\epsilon_h\|_{  T}^2\Big)^{\frac{1}{2}} \Big(\sum_{T\in{\cal T}_h}  h_T^4 \|\nabla_w (Q_h\xi-\xi)\|_{ T}^2\Big)^{\frac{1}{2}}\\&+\Big(\sum_{T\in{\cal T}_h}   h_T^4 \| \nabla_w\epsilon_h\|_{  T}^2\Big)^{\frac{1}{2}} \Big(\sum_{T\in{\cal T}_h}  h_T^4 \|\nabla_w  \xi\|_{ T}^2\Big)^{\frac{1}{2}}\\
\leq &   \3bar\epsilon_h\3bar_{W_h}  \3bar Q_h\xi-\xi\3bar_{W_h} +\3bar\epsilon_h\3bar_{W_h} \Big(\sum_{T\in{\cal T}_h}  h_T^4 \|{\cal Q}^{r_1}\nabla  \xi\|_{ T}^2\Big)^{\frac{1}{2}}\\
  \leq & Ch^{k-1} (\|\bu\|_{k+1}+\|(\nabla \times)^2\bu\|_{k-1})h^2\|\xi\|_1 \\
  \leq & Ch^{k+1} (\|\bu\|_{k+1}+\|(\nabla \times)^2\bu\|_{k-1}) \|\xi\|_1. 
  \end{split}
\end{equation*}
      
    Regarding to the term $I_8$,   using Cauchy-Schwarz inequality,  
    \eqref{erroresti1} with $k=0$, \eqref{regu},   we have
 \begin{equation*}  
     \begin{split}
       & \sum_{T\in{\cal T}_h} (\nabla_w  (\xi-Q_h\xi), \be_0)_T \\\leq &  \Big(\sum_{T\in{\cal T}_h}   \|  \nabla_w  (\xi-Q_h\xi)\|_{  T}^2\Big)^{\frac{1}{2}} \Big(\sum_{T\in{\cal T}_h}   \|\be_0\|_{ T}^2\Big)^{\frac{1}{2}}\\\leq& C h^2\|\xi\|_1\|
        \be_0\|\\
        \leq &  C h^2\|\xi\|_1(\|
        \bzeta_0\|+\|\bu-Q_0\bu\|)\\
        \leq & C h^2\|\xi\|_1 \|
\bzeta_0\|+Ch^2\|\xi\|_1h^{k+1}\|\bu\|_{k+1}
\\
\leq & h^2  \|
\bzeta_0\|^2+C \|\xi\|_1h^{k+3}\|\bu\|_{k+1}.
     \end{split}
    \end{equation*}

  Regarding to the term $I_9$, using Cauchy-Schwarz inequality, \eqref{l-2}, we have
\begin{equation*} \label{last2}
     \begin{split}
        |\sum_{T\in{\cal T}_h} (\nabla_w  \xi, \bQ_0\bu-\bu)_T |=& |\sum_{T\in{\cal T}_h} ({\cal Q}_h^{r_1}\nabla \xi, \bQ_0\bu-\bu)_T 
     |\\
     \leq & \Big(\sum_{T\in{\cal T}_h}   \|  {\cal Q}_h^{r_1}\nabla \xi\|_{  T}^2\Big)^{\frac{1}{2}} \Big(\sum_{T\in{\cal T}_h}   \|\bQ_0\bu-\bu\|_{ T}^2\Big)^{\frac{1}{2}}\\
     \leq & Ch^{k+1}\|\bu\|_{k+1}\|\xi\|_1.
     \end{split}
    \end{equation*}
    
      Substituting  the estimates of $I_i$ for $i=1,\cdots,9$ into 
      \eqref{dq} and using the regularity assumption \eqref{regu}     gives
      \begin{equation*}
          \begin{split}
           (1-h^2)    \|\bzeta_0\|^2 \leq & Ch^{t_0+k-2} (\|\bu\|_{k+1}+\|(\nabla \times)^2\bu\|_{k-1})\|\bphi\|_{t_0+1}   \\&   +Ch^{t_0+k-2}  \|(\nabla \times)^2\bu\|_{k-1} \|\bphi\|_{t_0+1}   \\&+Ch^{k+t_0-2}\|\bphi\|_{t_0+1} (\|\bu\|_{k+1}+\|(\nabla \times)^2\bu\|_{k-1})\\&
        +Ch^{k+t_0-2}\|\bu\|_{k+1} \|(\nabla\times)^2\bphi\|_{t_0-1}\\&+C h^{k+t_0-2}\|(\nabla \times)^2\bphi\|_{t_0-1}  (\|\bu\|_{k+1}+\|(\nabla\times)^2\bu\|_{k-1}) \\   &    +Ch^{k+1}\|\bu\|_{k+1}\|\xi\|_1\\    &  +Ch^{k+1}(\|\bu\|_{k+1}+\|(\nabla\times)^2\bu\|_{k-1})\|\xi\|_1 \\    &+Ch^{k+3}\|\bu\|_{k+1}\|\xi\|_1+Ch^{k+1}\|\bu\|_{k+1}\|\xi\|_1,
          \end{split}
      \end{equation*} 
      which gives
      $$
       \|\bzeta_0\|\leq Ch^{t_0+k-2} (\|\bu\|_{k+1}+\|(\nabla \times)^2\bu\|_{k-1}),
      $$
provided that $h$ is sufficiently small such that $1-h^2>0$. 
This, together with the triangle inequality, we have
      \begin{equation*}
      \begin{split}
            \|\be_0\|\leq & \|\bzeta_0\|+\|\bu-\bQ_0\bu\|\\\leq &   Ch^{t_0+k-2} (\|\bu\|_{k+1}+\|(\nabla \times)^2\bu\|_{k-1})+Ch^{k+1}\|\bu\|_{k+1}\\\leq &  Ch^{t_0+k-2} (\|\bu\|_{k+1}+\|(\nabla \times)^2\bu\|_{k-1}).
      \end{split}
      \end{equation*}
    
      This completes the proof of the theorem.
 \end{proof}

\section{Numerical test}
In the first test,  we solve \eqref{weakform} on the unit square domain $\Omega
  =(0,1)\times(0,1)$.
We choose the exact solution as
\an{\label{e1} \b u=\b{rot} 2^8 (x-x^2)^3 (y-y^2)^3, \
                p=0.  }

     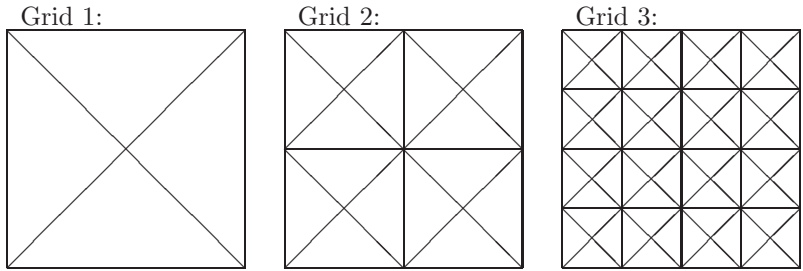
\begin{figure}[H] \setlength\unitlength{1pt}\begin{center}
    \begin{picture}(300,100)(0,0)
     \def\mc{\begin{picture}(90,90)(0,0)
       \put(0,0){\line(1,0){90}} \put(0,90){\line(1,0){90}}
      \put(0,0){\line(0,1){90}}  \put(90,0){\line(0,1){90}} 
       \put(0,0){\line(1,1){90}}  \put(0,90){\line(1,-1){90}}  
      \end{picture}}

    \put(0,0){\mc} \put(5,92){Grid 1:} \put(110,92){Grid 2:}\put(215,92){Grid 3:}
      \put(105,0){\setlength\unitlength{0.5pt}\begin{picture}(90,90)(0,0)
    \put(0,0){\mc}\put(90,0){\mc}\put(0,90){\mc}\put(90,90){\mc}\end{picture}}
      \put(210,0){\setlength\unitlength{0.25pt}\begin{picture}(90,90)(0,0)
    \multiput(0,0)(90,0){4}{\multiput(0,0)(0,90){4}{\mc}} \end{picture}}
    \end{picture}
 \caption{The first three grids for the computation in
    Tables \ref{t-1}--\ref{t-2}. }\label{grid1} 
    \end{center} \end{figure}

We compute the finite element solutions for \eqref{e1} on uniform 
       triangular grids shown in Figure \ref{grid1} by 
  the  $P_k$/$P_{k-1}$ WG finite elements 
      for $k=2$ and $3$.
The results are listed in Tables \ref{t-1}--\ref{t-2}. 
The optimal order of convergence is achieved in all cases.

\begin{table}[H]
  \centering  \renewcommand{\arraystretch}{1.1}
  \caption{The error and the computed order of convergence 
     for the solution \eqref{e1} on Figure \ref{grid1} meshes. }
  \label{t-1}
\begin{tabular}{c|cc|cc|cc}
\hline
Grid &  $\|\b u-\b u_0\|_0$  & $O(h^r)$ & $\3bar\b u-\b u_h\3bar$  & $O(h^r)$ &
      $\|p -p_0\|_0$  & $O(h^r)$    \\
\hline&\multicolumn{6}{c}{ By the $P_2$/$P_{1}$ WG finite element.}\\
\hline 
 1&    0.725E+03 &  0.0&    0.401E+02 &  0.0&    0.400E+02 &  0.0 \\
 2&    0.152E+02 &  5.6&    0.255E+02 &  0.7&    0.124E+01 &  5.0 \\
 3&    0.116E+01 &  3.7&    0.193E+02 &  0.4&    0.329E+00 &  1.9 \\
 4&    0.398E-01 &  4.9&    0.106E+02 &  0.9&    0.236E-01 &  3.8 \\
 5&    0.197E-02 &  4.3&    0.543E+01 &  1.0&    0.161E-02 &  3.9 \\
 6&    0.379E-03 &  2.4&    0.273E+01 &  1.0&    0.104E-03 &  3.9 \\
 7&    0.944E-04 &  2.0&    0.137E+01 &  1.0&    0.144E-04 &  2.9 \\
\hline
    \end{tabular}%
\end{table}%

\begin{table}[H]
  \centering  \renewcommand{\arraystretch}{1.1}
  \caption{The error and the computed order of convergence 
     for the solution \eqref{e1} on Figure \ref{grid1} meshes. }
  \label{t-2}
\begin{tabular}{c|cc|cc|cc}
\hline
Grid &  $\|\b u-\b u_0\|_0$  & $O(h^r)$ & $\3bar\b u-\b u_h\3bar$  & $O(h^r)$ &
      $\|p -p_0\|_0$  & $O(h^r)$    \\
\hline&\multicolumn{6}{c}{ By the $P_3$/$P_{2}$ WG finite element.}\\
\hline 
 1&    0.138E+03 &  0.0&    0.255E+02 &  0.0&    0.103E+02 &  0.0 \\
 2&    0.234E+01 &  5.9&    0.245E+02 &  0.1&    0.113E+01 &  3.2 \\
 3&    0.102E-01 &  7.8&    0.622E+01 &  2.0&    0.491E-01 &  4.5 \\
 4&    0.807E-04 &  7.0&    0.162E+01 &  1.9&    0.224E-02 &  4.5 \\
 5&    0.455E-05 &  4.1&    0.409E+00 &  2.0&    0.141E-03 &  4.0 \\
\hline
    \end{tabular}%
\end{table}%

We next compute the finite element solutions for \eqref{e1} on non-convex polygonal 
        grids shown in Figure \ref{grid2} by 
  the  $P_k$/$P_{k-1}$ WG finite elements 
      for $k=2$ and $3$.
The results are listed in Tables \ref{t-3}--\ref{t-4}. 
The optimal order of convergence is achieved for all solutions in all norms.

     \begin{figure}[H] \setlength\unitlength{1.2pt}\begin{center}
    \begin{picture}(280,98)(0,0)
     \def\mc{\begin{picture}(90,90)(0,0)
       \put(0,0){\line(1,0){90}} \put(0,90){\line(1,0){90}}
      \put(0,0){\line(0,1){90}}  \put(90,0){\line(0,1){90}} 
      
       \put(30,75){\line(1,-2){30}}   
       \put(0,0){\line(2,5){30}}     
       \put(90,90){\line(-2,-5){30}}  
      \end{picture}}

    \put(0,0){\mc} \put(0,92){Grid 1:} \put(95,92){Grid 2:}  \put(190,92){Grid 3:}
    
      \put(95,0){\setlength\unitlength{0.6pt}\begin{picture}(90,90)(0,0)
    \put(0,0){\mc}\put(90,0){\mc}\put(0,90){\mc}\put(90,90){\mc}\end{picture}}
     \put(190,0){\setlength\unitlength{0.3pt}\begin{picture}(90,90)(0,0)
      \multiput(0,0)(90,0){4}{\multiput(0,0)(0,90){4}{\mc}} \end{picture}}
    \end{picture}
 \caption{The first three non-convex polygonal grids for the computation in
    Tables \ref{t-3}--\ref{t-4}. }\label{grid2} 
    \end{center} \end{figure}
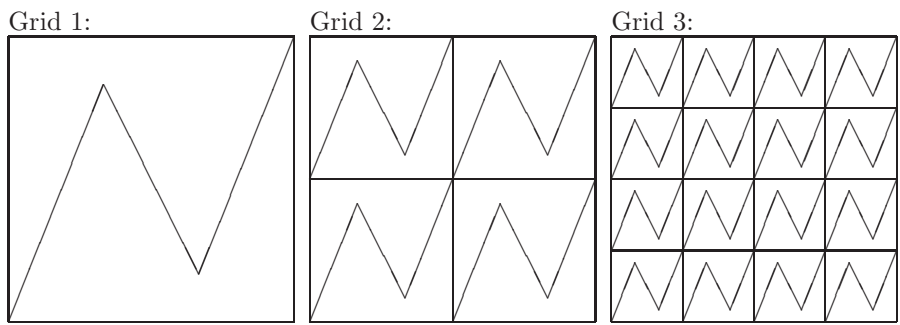

\begin{table}[H]
  \centering  \renewcommand{\arraystretch}{1.1}
  \caption{The error profile 
     for computing \eqref{e1} on Figure \ref{grid2} 
     non-convex polygonal meshes. }
  \label{t-3}
\begin{tabular}{c|cc|cc|cc}
\hline
Grid &  $\|\b u-\b u_0\|_0$  & $O(h^r)$ & $\3bar\b u-\b u_h\3bar$  & $O(h^r)$ &
      $\|p -p_0\|_0$  & $O(h^r)$    \\
\hline&\multicolumn{6}{c}{ By the $P_2$/$P_{1}$ WG finite element.}\\
\hline  
 1&    0.194E+04 &  0.0&    0.246E+03 &  0.0&    0.212E+03 &  0.0 \\
 2&    0.108E+03 &  4.2&    0.175E+03 &  0.5&    0.544E+02 &  2.0 \\
 3&    0.290E+01 &  5.2&    0.164E+03 &  0.1&    0.513E+01 &  3.4 \\
 4&    0.588E-01 &  5.6&    0.933E+02 &  0.8&    0.407E+00 &  3.7 \\
 5&    0.169E-02 &  5.1&    0.482E+02 &  1.0&    0.271E-01 &  3.9 \\
\hline
    \end{tabular}%
\end{table}%

\begin{table}[H]
  \centering  \renewcommand{\arraystretch}{1.1}
  \caption{The error profile 
     for computing \eqref{e1} on Figure \ref{grid2} 
     non-convex polygonal meshes. }
  \label{t-4}
\begin{tabular}{c|cc|cc|cc}
\hline
Grid &  $\|\b u-\b u_0\|_0$  & $O(h^r)$ & $\3bar\b u-\b u_h\3bar$  & $O(h^r)$ &
      $\|p -p_0\|_0$  & $O(h^r)$    \\
\hline&\multicolumn{6}{c}{ By the $P_3$/$P_{2}$ WG finite element.}\\
\hline  
 1&    0.103E+05 &  0.0&    0.122E+04 &  0.0&    0.887E+03 &  0.0 \\
 2&    0.126E+03 &  6.4&    0.440E+03 &  1.5&    0.483E+02 &  4.2 \\
 3&    0.215E+01 &  5.9&    0.155E+03 &  1.5&    0.272E+01 &  4.2 \\
 4&    0.181E-01 &  6.9&    0.431E+02 &  1.8&    0.918E-01 &  4.9 \\
 5&    0.145E-03 &  7.0&    0.111E+02 &  2.0&    0.309E-02 &  4.9 \\ 
\hline
    \end{tabular}%
\end{table}%

In the 3D numerical computation,  the domain for problem \eqref{weakform}
   is the unit cube $\Omega=(0,1)\times(0,1)\times(0,1)$.
We choose an $\b f$ and a $g$ in \eqref{weakform} so that the exact solution is
\an{\label{e3} \ad{  
  \b u &=\p{e^y \\ e^z \\ e^x },\qquad
            p&=0. }  }

We   compute the finite element solutions for \eqref{e3} on the
        grids shown in Figure \ref{grid4} by 
  the  $P_k$/$P_{k-1}$ WG finite elements 
      for $k=2$ and $3$.
The results are listed in Tables  \ref{t-5}--\ref{t-6}. 
The optimal order of convergence is achieved for all solutions in all norms.

\begin{figure}[H] 
\begin{center}
 \setlength\unitlength{1pt}
    \begin{picture}(320,118)(0,3)
    \put(0,0){\begin{picture}(110,110)(0,0) \put(25,102){Grid 1:}
       \multiput(0,0)(80,0){2}{\line(0,1){80}}  \multiput(0,0)(0,80){2}{\line(1,0){80}}
       \multiput(0,80)(80,0){2}{\line(1,1){20}} \multiput(0,80)(20,20){2}{\line(1,0){80}}
       \multiput(80,0)(0,80){2}{\line(1,1){20}}  \multiput(80,0)(20,20){2}{\line(0,1){80}}
    \put(80,0){\line(-1,1){80}}\put(80,0){\line(1,5){20}}\put(80,80){\line(-3,1){60}}
      \end{picture}}
    \put(110,0){\begin{picture}(110,110)(0,0)\put(25,102){Grid 2:}
       \multiput(0,0)(40,0){3}{\line(0,1){80}}  \multiput(0,0)(0,40){3}{\line(1,0){80}}
       \multiput(0,80)(40,0){3}{\line(1,1){20}} \multiput(0,80)(10,10){3}{\line(1,0){80}}
       \multiput(80,0)(0,40){3}{\line(1,1){20}}  \multiput(80,0)(10,10){3}{\line(0,1){80}}
    \put(80,0){\line(-1,1){80}}\put(80,0){\line(1,5){20}}\put(80,80){\line(-3,1){60}}
       \multiput(40,0)(40,40){2}{\line(-1,1){40}} 
        \multiput(80,40)(10,-30){2}{\line(1,5){10}}
        \multiput(40,80)(50,10){2}{\line(-3,1){30}}
      \end{picture}}
    \put(220,0){\begin{picture}(110,110)(0,0) \put(25,102){Grid 3:}
       \multiput(0,0)(20,0){5}{\line(0,1){80}}  \multiput(0,0)(0,20){5}{\line(1,0){80}}
       \multiput(0,80)(20,0){5}{\line(1,1){20}} \multiput(0,80)(5,5){5}{\line(1,0){80}}
       \multiput(80,0)(0,20){5}{\line(1,1){20}}  \multiput(80,0)(5,5){5}{\line(0,1){80}}
    \put(80,0){\line(-1,1){80}}\put(80,0){\line(1,5){20}}\put(80,80){\line(-3,1){60}}
       \multiput(40,0)(40,40){2}{\line(-1,1){40}} 
        \multiput(80,40)(10,-30){2}{\line(1,5){10}}
        \multiput(40,80)(50,10){2}{\line(-3,1){30}}

       \multiput(20,0)(60,60){2}{\line(-1,1){20}}   \multiput(60,0)(20,20){2}{\line(-1,1){60}} 
        \multiput(80,60)(15,-45){2}{\line(1,5){5}} \multiput(80,20)(5,-15){2}{\line(1,5){15}}
        \multiput(20,80)(75,15){2}{\line(-3,1){15}}\multiput(60,80)(25,5){2}{\line(-3,1){45}}
      \end{picture}}

    \end{picture} 
    \end{center} 
\caption{ The first three grids for the computation 
    in Tables \ref{t-5}--\ref{t-6}.  } 
\label{grid4}
\end{figure}
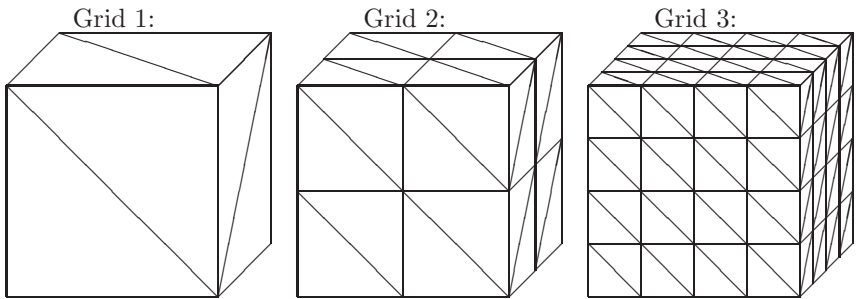

\begin{table}[H]
  \centering  \renewcommand{\arraystretch}{1.1}
  \caption{The error profile 
     for computing \eqref{e3} on Figure \ref{grid4} 
     tetrahedral meshes. }
  \label{t-5}
\begin{tabular}{c|cc|cc|cc}
\hline
Grid &  $\|\b u-\b u_0\|_0$  & $O(h^r)$ & $\3bar\b u-\b u_h\3bar$  & $O(h^r)$ &
      $\|p -p_0\|_0$  & $O(h^r)$    \\
\hline&\multicolumn{6}{c}{ By the $P_2$/$P_{1}$ WG finite element.}\\
\hline   
 1 &   0.115E-01 &0.00 &   0.377E+01 &0.00 &   0.137E-03 &0.00 \\
 2 &   0.326E-02 &1.82 &   0.137E+01 &1.46 &   0.442E-04 &1.63 \\
 3 &   0.575E-03 &2.50 &   0.507E+00 &1.44 &   0.571E-05 &2.95 \\
 4 &   0.112E-03 &2.36 &   0.226E+00 &1.16 &   0.730E-06 &2.97 \\
\hline
    \end{tabular}%
\end{table}%

\begin{table}[H]
  \centering  \renewcommand{\arraystretch}{1.1}
  \caption{The error profile 
     for computing \eqref{e3} on Figure \ref{grid4} 
     tetrahedral meshes. }
  \label{t-6}
\begin{tabular}{c|cc|cc|cc}
\hline
Grid &  $\|\b u-\b u_0\|_0$  & $O(h^r)$ & $\3bar\b u-\b u_h\3bar$  & $O(h^r)$ &
      $\|p -p_0\|_0$  & $O(h^r)$    \\
\hline&\multicolumn{6}{c}{ By the $P_3$/$P_{2}$ WG finite element.}\\
\hline    
 1 &   0.168E+00 &0.00 &   0.941E+02 &0.00 &   0.347E-02 &0.00 \\
 2 &   0.152E-01 &3.47 &   0.113E+02 &3.05 &   0.352E-03 &3.30 \\
 3 &   0.120E-02 &3.65 &   0.139E+01 &3.02 &   0.253E-04 &3.80 \\
\hline
    \end{tabular}%
\end{table}%

Finally, we compute the finite element solutions for \eqref{e3} on the
   non-convex polyhedral     grids shown in Figure \ref{grid3-2} by 
  the  $P_2$/$P_{1}$ WG finite element method.
The results are listed in Table \ref{t-7}. 
The optimal order of convergence is achieved for all solutions in all norms.

\begin{figure}[H] 
\begin{center}
 \setlength\unitlength{1pt}
    \begin{picture}(320,118)(0,3)
    \put(0,0){\begin{picture}(110,110)(0,0) \put(25,102){Grid 1:}
       \multiput(0,0)(80,0){2}{\line(0,1){80}}  \multiput(0,0)(0,80){2}{\line(1,0){80}}
       \multiput(0,80)(80,0){2}{\line(1,1){20}} \multiput(0,80)(20,20){2}{\line(1,0){80}}
       \multiput(80,0)(0,80){2}{\line(1,1){20}}  \multiput(80,0)(20,20){2}{\line(0,1){80}}

    \multiput(0,0)(2,2){10}{\circle*{1}}
    \multiput(60,40)(2,2){10}{\circle*{1}}
    \multiput(20,20)(0,2){40}{\circle*{1}}
        \multiput(20,20)(2,0){40}{\circle*{1}}
        
        \multiput(40,100)(2,-2){20}{\circle*{1}}
        
    \put(20,0){\line(1,1){40}}\put(20,80){\line(1,-1){40}}
    \put(20,80){\line(1,1){20}}
      \end{picture}}

    \put(110,0){\begin{picture}(110,110)(0,0)\put(25,102){Grid 2:}
       \multiput(0,0)(40,0){3}{\line(0,1){80}}  \multiput(0,0)(0,40){3}{\line(1,0){80}}
       \multiput(0,80)(40,0){3}{\line(1,1){20}} \multiput(0,80)(10,10){3}{\line(1,0){80}}
       \multiput(80,0)(0,40){3}{\line(1,1){20}}  \multiput(80,0)(10,10){3}{\line(0,1){80}}

    \multiput(10,0)(40,0){2}{\multiput(0,0)(0,40){2}{\line(1,1){20}}}
    \multiput(10,40)(40,0){2}{\multiput(0,0)(0,40){2}{\line(1,-1){20}}}
    \multiput(10,80)(40,0){2}{\line(1,1){20}}
      \end{picture}}
    \put(220,0){\begin{picture}(110,110)(0,0) \put(25,102){Grid 3:}
       \multiput(0,0)(20,0){5}{\line(0,1){80}}  \multiput(0,0)(0,20){5}{\line(1,0){80}}
       \multiput(0,80)(20,0){5}{\line(1,1){20}} \multiput(0,80)(5,5){5}{\line(1,0){80}}
       \multiput(80,0)(0,20){5}{\line(1,1){20}}  \multiput(80,0)(5,5){5}{\line(0,1){80}}

    \multiput(5,0)(20,0){4}{\multiput(0,0)(0,20){4}{\line(1,1){10}}}
    \multiput(5,20)(20,0){4}{\multiput(0,0)(0,20){4}{\line(1,-1){10}}}
    \multiput(5,80)(20,0){4}{\line(1,1){20}}
    
      \end{picture}}

    \end{picture} 
    \end{center} 
\caption{ The first three (non convex polyhedral) 
  grids for the computation 
    in Table  \ref{t-7}.  } 
\label{grid3-2}
\end{figure}
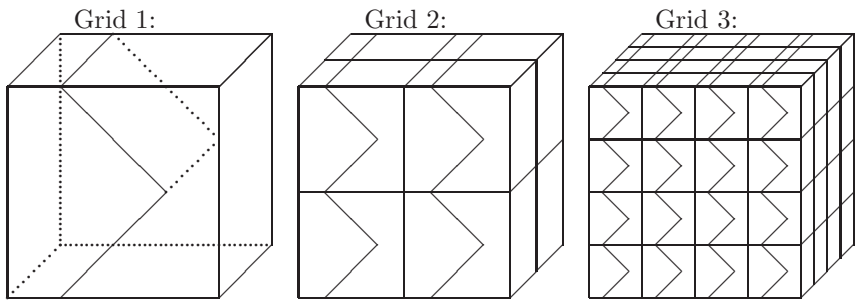

\begin{table}[H]
  \centering  \renewcommand{\arraystretch}{1.1}
  \caption{The error profile 
     for computing \eqref{e3} on Figure \ref{grid3-2} 
     non-convex polyhedral meshes. }
  \label{t-7}
\begin{tabular}{c|cc|cc|cc}
\hline
Grid &  $\|\b u-\b u_0\|_0$  & $O(h^r)$ & $\3bar\b u-\b u_h\3bar$  & $O(h^r)$ &
      $\|p -p_0\|_0$  & $O(h^r)$    \\
\hline&\multicolumn{6}{c}{ By the $P_2$/$P_{1}$ WG finite element.}\\
\hline    
 1 &   0.168E+00 &0.00 &   0.941E+02 &0.00 &   0.347E-02 &0.00 \\
 2 &   0.152E-01 &3.47 &   0.113E+02 &3.05 &   0.352E-03 &3.30 \\
 3 &   0.120E-02 &3.65 &   0.139E+01 &3.02 &   0.253E-04 &3.80 \\
\hline
    \end{tabular}%
\end{table}%

\bibliographystyle{abbrv}
\bibliography{Ref}

\end{document}